\documentclass[10pt]{article}
\usepackage{amsthm}
\usepackage{amssymb}
\usepackage{amsmath}
\usepackage{mathrsfs}
\usepackage{graphicx}
\usepackage{comment}
\usepackage{float}
\usepackage{mathtools}
\usepackage[backend=bibtex,style=numeric,sorting=nyt]{biblatex}
\mathtoolsset{showonlyrefs}
\allowdisplaybreaks
\title{Splitting Quantum Graphs}
\date{\today}
\author{Nathaniel Smith and Alim Sukhtayev}

\newcommand{\CC}{{\mathbb{C}}}

\newcommand{\RR}{{\mathbb{R}}}

\newcommand{\diag}{{\rm diag \,}}

\newcommand{\rlam}{{\sqrt{\lambda}}}

\newcommand{\dom}{{\text{dom}\,}}
\newcommand{\ran}{{\text{ran}\,}}
\newcommand{\vy}{{\Vec{y}}}
\newcommand{\vz}{{\Vec{z}}}
\newcommand{\vw}{{\Vec{w}}}
\newcommand{\vu}{{\Vec{u}}}
\newcommand{\vv}{{\Vec{v}}}
\newcommand{\vf}{{\Vec{f}}}
\newcommand{\vx}{{\Vec{x}}}

\newcommand{\vc}{{\Vec{c}}}
\newcommand{\U}{{\mathcal{U}}}
\newcommand{\E}{{\mathcal{E}}}
\newcommand{\B}{{\mathcal{B}}}
\newcommand{\la}{{\langle}}
\newcommand{\ra}{\rangle}
\newcommand{\M}{\mathcal{M}}

\newcommand{\sqrten}{\sqrt{\lambda -\nu}}
\newcommand{\myeq}[1]{\mathrel{\stackrel{\makebox[0pt]{\mbox{\normalfont\tiny {#1}}}}{=}}}
\newtheorem{theorem}{Theorem}[section]

\newtheorem{lemma}[theorem]{Lemma}

\theoremstyle{definition}
\newtheorem{definition}[theorem]{Definition}
\newtheorem{remark}[theorem]{Remark}

\begin{document}
\maketitle
\begin{abstract}
We derive a counting formula for the eigenvalues of Schr\"odinger operators with self-adjoint boundary conditions on quantum star graphs. More specifically, we develop techniques using Evans functions to reduce full quantum graph eigenvalue problems into smaller subgraph eigenvalue problems. These methods provide a simple way to calculate the spectra of operators with localized potentials.
\end{abstract}

\section{Introduction}
This paper studies the eigenvalues of Schr\"odinger operators on quantum star graphs with $n$ edges. We will refer to the $i$th edge of a quantum graph $\Omega$ as $\epsilon_i$. Each edge is parameterized to have coordinates between $0$ and $\ell_i>0$, where $0$ is assigned to the shared origin of all edges and $\ell_i$ is the other endpoint.

A function on a quantum graph can be written as a vector function $\vf$, where the $i$th component $f_i$ represents the portion of the function defined over the $i$th edge. We assign a separate variable $x_i$ to each edge $\epsilon_i$, which takes values on the interval $[0,\ell_i]$. The collection of all these $x_i$ is abbreviated as $\vx$.

\subsection{Basic Definitions and Notations}
We are interested particularly in sets of separated and self-adjoint boundary conditions $\Gamma=\{\alpha_1,\alpha_2,\beta_1,\beta_2\}$, where these $\alpha_i$ and $\beta_i$ are $n\times n$ matrices satisfying the following:
\begin{equation} \label{eq:3}
\begin{split}
\text{rank}(\begin{bmatrix}
\alpha_1 & \alpha_2
\end{bmatrix}) = \text{rank}(\begin{bmatrix}
\beta_1 & \beta_2
\end{bmatrix})=n\,,
\\
\alpha_1\alpha_2^* = \alpha_2\alpha_1^* \hspace{0.25in} \text{and} \hspace{0.25in} \beta_1\beta_2^* = \beta_2\beta_1^*\,,
\end{split}
\end{equation}
where the $*$ superscript represents the matrix adjoint. We say a function $\vu$ satisfies $\Gamma$ boundary conditions if it satisfies the equations
\begin{equation} \label{eq:5}
\begin{split}
\alpha_1 \vu(\Vec{0}) + \alpha_2 \vu'(\Vec{0}) = \Vec{0}\,,
\\
\beta_1 \vu(\Vec{\ell}) + \beta_2 \vu'(\Vec{\ell}) = \Vec{0}\,.
\end{split}
\end{equation}
In this paper, $\vu'$ refers to a vector whose $i$th component is $u_i'$, which is defined to be the derivative of the $i$th component $u_i$ of $\vu$ with respect to the relevant spatial variable $x_i$.

We will follow the reasonable convention that the matrices $\beta_1$ and $\beta_2$ should both be diagonal. This comes from the natural assumption that the only interaction between functions should happen at shared vertices (in this case, only at the origin).

\begin{definition}\label{def:trace}
Given a set $\Gamma$ of separated and self-adjoint boundary conditions as described above, we define the $\Gamma$-trace of a function $\vu$ as the following vector in $\CC^{2n}$:
\begin{equation} \label{eq:7}
\gamma_\Gamma \vu = \begin{bmatrix}
\beta_1 \vu(\Vec{\ell}) + \beta_2 \vu'(\Vec{\ell})\\
\alpha_1 \vu(\Vec{0}) + \alpha_2 \vu'(\Vec{0})
\end{bmatrix}\,,
\end{equation}
provided $\vu$ is sufficiently smooth.
\end{definition}
As an example, by choosing $\alpha_1$ and $\beta_1$ to be identity matrices while letting $\alpha_2$ and $\beta_2$ be zero matrices, $\Gamma$ would be equivalent to Dirichlet boundary conditions at all vertices. The Dirichlet trace would then be the column vector
\begin{equation} \label{eq:Dtrace}
\gamma_D \vu = \begin{bmatrix}
u_1(\ell_1)& \hdots &u_n(\ell_n) & u_1(0)& \hdots & u_n(0)
\end{bmatrix}^t\,,
\end{equation}
where $t$ represents the matrix (nonconjugating) transpose. One other notable trace that will show up often is the Neumann trace, defined as
\begin{equation} \label{eq:Ntrace}
\gamma_N\vu = \begin{bmatrix}
u_1'(\ell_1) & \hdots & u_n'(\ell_n) & -u_1'(0)& \hdots & -u_n'(0)
\end{bmatrix}^t\,.
\end{equation}

We are interested in solving eigenvalue problems of the form
\begin{equation} \label{eq:1}
\begin{cases}
(H^\Omega-\lambda) \vu = \Vec{0}\,,\\
\gamma_\Gamma \vu = \Vec{0}\,,
\end{cases}
\end{equation}
where $\vu(\vx,\lambda)$ is a function over $\Omega$, $H^\Omega$ is a differential expression defined as
\begin{equation} \label{eq:2}
H^\Omega = -\Delta + V := \diag (-\partial^2_{x_i}+V_i(x_i)| 1\leq i\leq n)\,,
\end{equation}
where $\diag$ denotes a diagonal matrix whose $i$th diagonal entry is the $i$th argument specified. We assume that each potential function $V_i$ is in $L^1(0,\ell_i)$.

\begin{definition}\label{def:lgamma}
$H^\Omega_\Gamma$ is a Schr\"odinger operator which is defined over the domain $\dom(H^\Omega_\Gamma):=\{\vu\in H^2(\Omega):\gamma_\Gamma\vu = \Vec{0}\}$, where $\vu\in H^2(\Omega)$ is shorthand to denote $(\vu)_i\in H^2(0,\ell_i)$ for all $i$, and $H^2(0,\ell_i)$ is a Sobolev space defined over $\epsilon_i$. Then, for functions in this domain,
\[
H^\Omega_\Gamma \vu := H^\Omega\vu.
\]
We denote the spectrum of $H^{\Omega}_\Gamma$ as $\sigma(H^\Omega_\Gamma)$, and the resolvent set as $\rho(H^{\Omega}_{\Gamma})$. Note that our assumptions on the potentials $V_i$ in $H^\Omega$ guarantee that $\sigma(H^{\Omega}_\Gamma)$ consists only of the discrete spectrum \cite{W}.
\end{definition}

\begin{remark}\label{rem:extraops}
We can also introduce smaller associated operators over subgraphs. Let $x_j=s_1$ be a splitting point (though not a vertex) on edge $\epsilon_j$, which will play the role of a vertex in two connected subgraphs of $\Omega$, denoted $\Omega_1$ and $\Omega_2$, which satisfy following conditions:
\[
\begin{split}
\Omega_1\cup \Omega_2 = \Omega &\text{  and  }\Omega_1\cap \Omega_2 = \{s_1\},
\\\ell_j\in \Omega_1 &\text{  and  } 0\in \Omega_2.
\end{split}
\]
An example of such a partition is shown in Figure \ref{fig:exampleSingle}. Let $H^{\Omega_1}$ and $H^{\Omega_2}$ be the restrictions of $H^{\Omega}$ to $\Omega_1$ and $\Omega_2$, respectively. Next, define a new set of boundary conditions $\Gamma_1$ over $\Omega_1$ which shares $\Gamma$ boundary conditions at all common vertices with $\Omega$, and which includes a Dirichlet boundary condition $u(s_1) = 0$ at $s_1$. Similarly, we define $\Gamma_2$ as a set of boundary conditions over $\Omega_2$ which is the same as $\Gamma$ at all common vertices with $\Omega$, and which includes the Dirichlet condition $u_j(s_1) = 0$ at $s_1$.
\begin{figure}[t]
    \centering
    \includegraphics[width=1.5in]{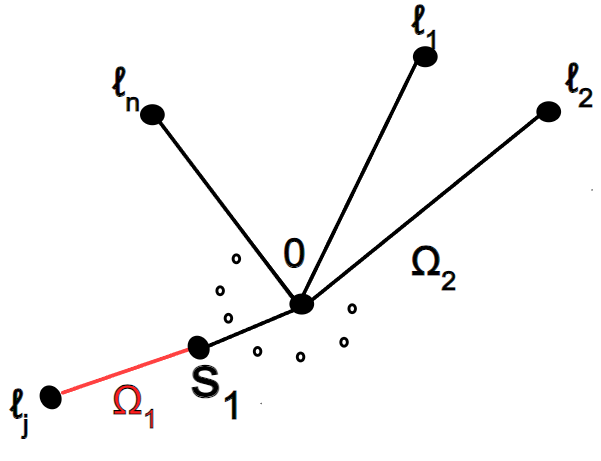}
    \caption{A partition of $\Omega$ into $\Omega_1$ and $\Omega_2$ at $s_1$.}
    \label{fig:qgraph}
\end{figure}

Then, we will consider the operators $H^{\Omega_1}_{\Gamma_1}$ and $H^{\Omega_2}_{\Gamma_2}$, which are defined as follows:
\begin{equation}
\begin{split}
H^{\Omega_1}_{\Gamma_1}u := H^{\Omega_1}u\text{ for }u\in \dom (H^{\Omega_1}_{\Gamma_1}) := \{u\in H^2(\Omega_1): \gamma_{\Gamma_1}u = \Vec{0}\},
\\H^{\Omega_2}_{\Gamma_2}\vu := H^{\Omega_2}\vu\text{ for }\vu\in \dom (H^{\Omega_2}_{\Gamma_2}) := \{\vu\in H^2(\Omega_2): \gamma_{\Gamma_2}\vu = \Vec{0}\}.
\end{split}
\end{equation}
\end{remark}
Note that $\gamma_{\Gamma_1}u$ and $\gamma_{\Gamma_2}\vu$ have a near identical definition to that in formula \eqref{eq:7} for their respective $\alpha$ and $\beta$ matrices. However, in the case of $\gamma_{\Gamma_1}u$, the new vertex $s_1$ will take the place of the origin $0$, and in the case of $\gamma_{\Gamma_2}\vu$ $s_1$ will take the place of $\ell_j$.

We will make heavy use of the Evans function, a powerful and popular tool in the stability theory of traveling waves \cite{AGJ,CLS13,KP13, PW1, S}. In general, Evans functions are not uniquely defined. To avoid this issue, we specify a unique construction of an Evans function $E^{\Omega}_\Gamma$ associated with a given operator $H^{\Omega}_\Gamma$ in Definition \ref{def:evansfunction}. Thus, for the remainder of the paper, we will simply refer to \textit{the} Evans function of an operator, always assuming that it is constructed by Definition \ref{def:evansfunction}.

If $\Omega$ is split at some some non-vertex point $x_j = s_1$ into regions $\Omega_1$ and $\Omega_2$ as discussed in Remark $\ref{rem:extraops}$, then a one-sided map is defined to be a Dirichlet-to-Neumann map associated with $s_1$. In particular, we specify two such maps given our splitting.

\begin{definition}\label{def:boundmap}
Let $\lambda\in \rho(H^{\Omega_1}_{\Gamma_1})$. Fix $f\in \CC$. Then we define a map $M_1(\lambda):\CC\rightarrow\CC$ as
\begin{equation}\label{eq:252}
M_1(\lambda)f = -u'(s_1,\lambda),    
\end{equation}
where $u$ is the unique solution to the boundary value problem
\[
\begin{cases}
(H^{\Omega_1}_{\Gamma_1}-\lambda)u = 0,\\
u(s_1) = f,\\
\text{$u$ satisfies all other $\Gamma_1$ conditions.}
\end{cases}
\]
Due to the definition of $u$, we have that
\[\label{eq:300}
M_1(\lambda) = -\frac{u'(s_1,\lambda)}{u(s_1,\lambda)}.
\]
We suppress the $s_1$ dependence of the map in this paper. 

Similarly, for $\lambda\in \rho(H^{\Omega_2}_{\Gamma_2})$ and $f\in\CC$ we define a map $M_2(\lambda):\CC\rightarrow \CC$ as
\begin{equation}\label{eq:251}
M_2(\lambda)f = u_j'(s_1,\lambda),
\end{equation}
where $u_j$ is the $j$th component of the unique solution $\vu$ to the boundary problem
\[
\begin{cases}
(H^{\Omega_2}_{\Gamma_2}-\lambda)\vu = \Vec{0},\\
u_j(s_1) = f,\\
\text{$\vu$ satisfies all other $\Gamma_2$ conditions.}
\end{cases}
\]
Due to the definition of $\vu$, we can exactly express $M_2(\lambda)$ as
\[\label{eq:310}
M_2(\lambda) = \frac{u_j'(s_1,\lambda)}{u_j(s_1,\lambda)}.
\]
Again, we choose to suppress the $s_1$ dependence of this map. Note that the right hand sides of equations \eqref{eq:252} and \eqref{eq:251} are the portions of the Neumann trace over $\Omega_1$ and $\Omega_2$, respectively, which correspond to $s_1$. The sum $M_1+M_2$ is called a two-sided map associated with $s_1$.
\end{definition}

\subsection{Summary of Paper}
In this paper we develop methods for finding the spectra of Schr\"odinger differential operators on quantum star graphs. In particular, we wish to be able to split a quantum star graph into smaller subgraphs, and then find the spectrum of the original problem by examining the spectra of the related subgraph problems. This is particularly helpful for dealing with operators with localized potentials, since it allows us to isolate the zero potential and nonzero potential portions from one another, resulting in simpler subproblems.

In Section 2, we find a general formula for $R_\lambda^{\Gamma}\vv$, the resolvent of $H^\Omega_\Gamma$ applied to an arbitrary $L^2(\Omega)$ function $\vv$.

In Section 3, we express $\vu_{\Gamma,i}$ (a solution to an inhomogeneous $\Gamma$ boundary value problem) in terms of $R^\Gamma_\lambda\vv$. Such functions are the building blocks of the two-sided maps. This expression (Theorem \ref{thm:general_ugamma}) involves three key projection matrices: the Dirichlet, Neumann, and Robin projections. The various properties of these matrices allow us to algebraically apply boundary conditions to solutions in a general way, without relying on any knowledge of the specific type of condition (beyond being representable by the $\alpha$ and $\beta$ matrices discussed earlier).

In Section 4, we prove the main results of this paper. They relate the eigenvalues of $H^\Omega_{\Gamma}$ to those of the smaller operators that come from splitting $\Omega$ as discussed in Remark \ref{rem:extraops}. We state them here:
\begin{theorem}\label{thm:determinant_equivalence}
Suppose $\Omega$ is split into $\Omega_1$ and $\Omega_2$ at some non-vertex point $s_1$, and let $\Gamma_1$ and $\Gamma_2$ be the new boundary conditions as constructed in Remark \ref{rem:extraops}. Additionally, suppose $\lambda\in\rho(H^{\Omega_1}_{\Gamma_1})\cap \rho(H^{\Omega_2}_{\Gamma_2})$, and let $M_1+M_2$ be the two-sided map at $s_1$. Then
\begin{equation}\label{eq:something}
\frac{E^\Omega_\Gamma(\lambda)}{E^{\Omega_1}_{\Gamma_1}(\lambda)E^{\Omega_2}_{\Gamma_2}(\lambda)}=(M_1 + M_2)(\lambda),
\end{equation}
where $E^\Omega_\Gamma$, $E^{\Omega_1}_{\Gamma_1}$, and $E^{\Omega_2}_{\Gamma_2}$ are the Evans functions associated with $H^\Omega_\Gamma$, $H^{\Omega_1}_{\Gamma_1}$, and $H^{\Omega_2}_{\Gamma_2}$, respectively.
\end{theorem}
\begin{remark}
Despite the fact that the left and right sides of equation \eqref{eq:something} are not defined on the spectra of any of the operators involved, both sides are equal as meromorphic functions over $\CC$. Since we only care about the poles and zeros in this equation, this type of equality is entirely sufficient.
\end{remark}
A similar result has already been proven for the interval case in \cite{DCDS2012}.

Since zeros of Evans functions locate eigenvalues, Theorem \ref{thm:determinant_equivalence} can be used to easily prove anther one of the main results.
\begin{theorem}\label{thm:counting}
Let $\mathcal{N}_X(\lambda)$ count the number of eigenvalues less than or equal to $\lambda$ of any operator $X$, and let $N(\lambda)$ be the difference between the number of zeros and number of poles (including multiplicities) of $M_1+M_2$ less than or equal to $\lambda$. Then
\begin{equation}
\mathcal{N}_{H^\Omega_\Gamma}(\lambda)=\mathcal{N}_{H^{\Omega_1}_{\Gamma_1}}(\lambda)+\mathcal{N}_{H^{\Omega_2}_{\Gamma_2}}(\lambda)+N(\lambda).
\end{equation}
\end{theorem}
In \cite{MR3961373}, a similar eigenvalue counting theorem has already been proven for the case of dividing multidimensional domains by hypersurfaces and defining two-sided maps with respect to these hypersurfaces.

We will also demonstrate how repeated applications of Theorem \ref{thm:determinant_equivalence} can be used to split a quantum graph into three subdomains over which the eigenvalues can be counted separately, which is useful for more general barrier and well problems. 
\begin{definition}\label{def:secondsplit}
Suppose $\Omega$ has been split into $\Omega_1$ and $\Omega_2$ as described in Remark \ref{rem:extraops} at a non-vertex splitting point $x_j=s_1$. We can apply the exact same process to split $\Omega_2$ at some non-vertex point $x_i=s_2$ into two connected subgraphs $\Tilde{\Omega}_1$ and $\Tilde{\Omega}_2$ which satisfy the following conditions:
\[
\begin{split}
\Tilde{\Omega}_1\cup \Tilde{\Omega}_2 = \Omega_2 &\text{  and  }\Tilde{\Omega}_1\cap \Tilde{\Omega}_2 = \{s_2\}.
\end{split}
\]
Additionally, we ensure that $0\in \Tilde{\Omega}_2$. If $s_2$ is on the same wire as $s_1$, then $s_1\in \Tilde{\Omega}_1$. If the two splitting points are on different wires, then $\ell_i\in \Tilde{\Omega}_1$.
\end{definition}

\begin{remark}
Again, such a splitting comes with some associated operators. Let $H^{\Tilde{\Omega}_1}$ and $H^{\Tilde{\Omega}_2}$ be the restrictions of $H^{\Omega_2}$ to $\Tilde{\Omega}_1$ and $\Tilde{\Omega}_2$, respectively. We can define new boundary conditions $\Tilde{\Gamma}_1$ on $\Tilde{\Omega}_1$ which match $\Gamma_2$ conditions at all common vertices with $\Omega_2$, and which include the Dirichlet condition $u(s_2)=0$ at $x_i=s_2$. Likewise, we can define $\Tilde{\Gamma}_2$ boundary conditions over $\Tilde{\Omega}_2$ which are identical to $\Gamma_2$ conditions at all common vertices, and with the Dirichlet condition $u_i(s_2)=0$ at $s_2$. Finally, using these pieces we can define the operators $H^{\Tilde{\Omega}_1}_{\Tilde{\Gamma}_1}$ and $H^{\Tilde{\Omega}_2}_{\Tilde{\Gamma}_2}$ as follows:
\begin{equation}
\begin{split}
H^{\Tilde{\Omega}_1}_{\Tilde{\Gamma}_1}u := H^{\Tilde{\Omega}_1}u\text{ for }u\in \dom (H^{\Tilde{\Omega}_1}_{\Tilde{\Gamma}_1}) := \{u\in H^2(\Tilde{\Omega}_1): \gamma_{\Tilde{\Gamma}_1}u = \Vec{0}\},
\\H^{\Tilde{\Omega}_2}_{\Tilde{\Gamma}_2}\vu := H^{\Tilde{\Omega}_2}\vu\text{ for }\vu\in \dom (H^{\Tilde{\Omega}_2}_{\Tilde{\Gamma}_2}) := \{\vu\in H^2(\Tilde{\Omega}_2): \gamma_{\Tilde{\Gamma}_2}\vu = \Vec{0}\}.
\end{split}
\end{equation}
\end{remark}

We can also construct new Dirichlet-to-Neumann maps $\M_1$ and $\M_2$, now associated at two vertices $s_1$ and $s_2$, and sum them to make another two-sided map. The construction of $\M_1$ and $\M_2$ is different for splittings on the same wire vs splittings on different wires. The specific constructions are detailed in Definitions \ref{def:onewireMs} and \ref{def:twowireMs}. In either case, we can utilize the determinant of the two-sided map $|(\M_1+\M_2)(\lambda)|$ in the following theorem:

\begin{figure}[H]
    \centering
    \includegraphics[width=1.5in]{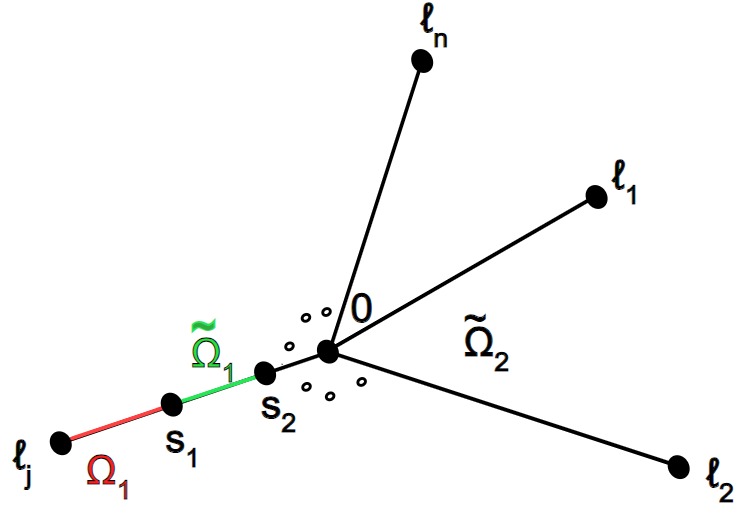}
    \includegraphics[width=1.5in]{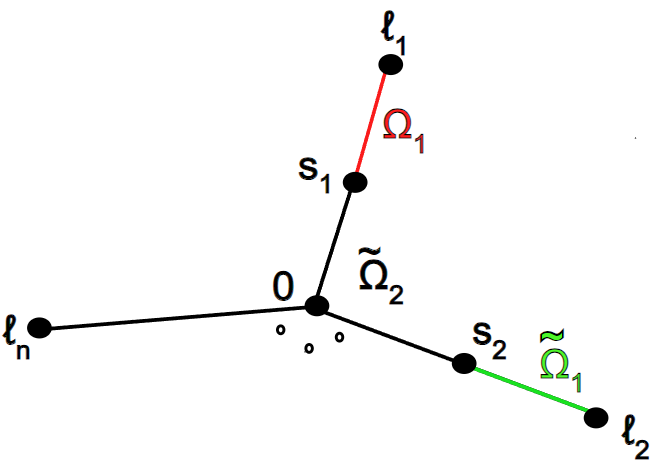}
    \caption{Two example partitions of $\Omega$ into three subgraphs.}
    \label{fig:twosplittwowire}
\end{figure}

\begin{theorem}\label{thm:twoCutEvans}
Suppose $\Omega$ is split into three quantum subgraphs $\Omega_1$, $\Tilde{\Omega}_1$, and $\Tilde{\Omega}_2$, as described in Definition \ref{def:secondsplit}. Additionally, suppose that $\lambda\in\rho(H^{\Omega_1}_{\Gamma_1})\cap \rho(H^{\Tilde{\Omega}_1}_{\Tilde{\Gamma}_1})\cap \rho(H^{\Tilde{\Omega}_2}_{\Tilde{\Gamma}_2})$, and let $\M_1+\M_2$ be the two-sided map associated with $s_1$ and $s_2$. Then
\begin{equation}
\frac{E^{\Omega}_{\Gamma}}{E^{\Omega_1}_{\Gamma_1}E^{\Tilde{\Omega}_1}_{\Tilde{\Gamma}_1}E^{\Tilde{\Omega}_2}_{\Tilde{\Gamma}_2}} = |(\M_1+\M_2)(\lambda)|,
\end{equation}
where $E^{\Omega}_{\Gamma}$, $E^{\Omega_1}_{\Gamma_1}$, $E^{\Tilde{\Omega}_1}_{\Tilde{\Gamma}_1}$, and $E^{\Tilde{\Omega}_2}_{\Tilde{\Gamma}_2}$ are the Evans functions associated with the operators $H^{\Omega}_{\Gamma}$, $H^{\Omega_1}_{\Gamma_1}$, $H^{\Tilde{\Omega}_1}_{\Tilde{\Gamma}_1}$, and $H^{\Tilde{\Omega}_2}_{\Tilde{\Gamma}_2}$, respectively.
\end{theorem}
Just as with Theorem \ref{thm:determinant_equivalence}, the equality in Theorem \ref{thm:twoCutEvans} is one of meromorphic functions, so the desired behavior is preserved. Also, we again obtain a corresponding eigenvalue counting theorem.
\begin{theorem}\label{thm:counting2d}
Let $\mathcal{N}_X(\lambda)$ count the number of eigenvalues less than or equal to $\lambda$ of any operator $X$, and let $N(\lambda)$ be the difference between the number of zeros and number of poles (including multiplicities) of $|\M_1+\M_2|$ less than or equal to $\lambda$. Then
\begin{equation}
\mathcal{N}_{H^\Omega_\Gamma}(\lambda)=\mathcal{N}_{H^{\Omega_1}_{\Gamma_1}}(\lambda)+\mathcal{N}_{H^{\Tilde{\Omega}_1}_{\Tilde{\Gamma}_1}}(\lambda)+\mathcal{N}_{H^{\Tilde{\Omega}_2}_{\Tilde{\Gamma}_2}}(\lambda)+N(\lambda).
\end{equation}
\end{theorem}

Finally, in Section 5, we illustrate some applications of the main theorems to quantum star graphs with potential barriers and wells.

\section{General Resolvent Formula}

\subsection{Motivation for Using the Resolvent}
While it is possible to construct Dirichlet-to-Neumann maps column-by-column by solving several boundary-value problems, this method does not do much to illuminate the processes by which the zeros and poles containing spectral information appear in the map. Provided that $\lambda\in \rho(H^\Omega_\Gamma)$, this map-making consists of finding functions $\vu_{\Gamma,i}\in H^2(\Omega)$ which serve as solutions to the problem:
\begin{equation} \label{eq:8}
\begin{cases}
(H^\Omega-\lambda)\vu_{\Gamma,i} = \Vec{0} \,,\\
\gamma_\Gamma \vu_{\Gamma,i} = \Vec{e}_i,
\end{cases}
\end{equation}
where $\Vec{e}_i$ is the $i$th standard basis vector for $\CC^{2n}$, then applying the desired trace to such a solution. Note that we will have $\Vec{e}_i$ representing standard basis vectors for $\CC^{n}$ and $\CC^{2n}$ interchangeably, since the dimension will be clear from context. Instead of solving for $\vu_{\Gamma,i}$ by its defining boundary-value problem, we will construct it using the resolvent of $H^\Omega_\Gamma$. 

For $\lambda\in \rho(H^{\Omega}_\Gamma)$, a key property of the resolvent $R_\lambda^\Gamma$ of $H^{\Omega}_\Gamma$ is that, for $\lambda\in \rho(H^{\Omega}_\Gamma)$, it acts as the inverse operator of $H^\Omega_\Gamma-\lambda$. That is,
\begin{equation}\label{eq:150}
R_\lambda^\Gamma(H^\Omega_\Gamma-\lambda )\vv = (H^\Omega_\Gamma-\lambda)R_\lambda^\Gamma\vv = \vv,
\end{equation}
for any function $\vv$ on $\Omega$.

Equation \eqref{eq:150} actually gives us a way to find $R_\lambda^\Gamma \vv$ for an arbitrary $L^2(\Omega)$ function $\vv$ by solving the inhomogeneous boundary-value problem:
\begin{equation} \label{eq:resolvent_def}
\begin{cases}
(H^\Omega_\Gamma - \lambda)\vy = \vv \,,
\\ \gamma_{\Gamma}\vy = \Vec{0}.
\end{cases}
\end{equation}
The unique solution $\vy$ to this problem must be the resolvent. We will solve for $R_\lambda^\Gamma\vv$ by adding together a complementary and particular solution. Normally, we could generate a particular solution via the method of Variation of Parameters, but since each component of a function on $\Omega$ has a different spatial variable, we need to make some slight adjustments to the method.

\subsection{Complementary Homogeneous Solution}
We will start by obtaining a complementary solution. First, suppose $\lambda\in \rho(H^\Omega_\Gamma)$. Next, let the entries of $\alpha_1$ and $\alpha_2$ be named as follows:
\begin{equation} \label{eq:9}
\alpha_1 = \begin{bmatrix}
a_{1,1} & \hdots & a_{1,n}\\
\vdots & \ddots & \vdots\\
a_{n,1} & \hdots & a_{n,n}
\end{bmatrix} \hspace{.5in} \alpha_2 = \begin{bmatrix}
b_{1,1} & \hdots & b_{1,n}\\
\vdots & \ddots & \vdots\\
b_{n,1} & \hdots & b_{n,n}
\end{bmatrix}\,,
\end{equation}
and similarly let the nonzero diagonal entries of $\beta_1$ and $\beta_2$ be named
\begin{equation} \label{eq:11}
\beta_1 = \diag (g_1,...,g_n) \hspace{.5in} \beta_2 = \diag (h_1,...,h_n)\,.
\end{equation}
Consider the matrix $\begin{bmatrix}
-\alpha_2^* \\ \alpha_1^*
\end{bmatrix}$. We can use each one of the columns as initial conditions for complementary homogeneous problems to get solutions $\vy_i$ which solve:
\begin{equation}\label{eq:10}
\begin{cases}
(H^\Omega-\lambda) \vy_i = \Vec{0}\,,\\
y_{i,j}(0) = -\overline{b_{i,j}} \text{ for $1\leq j\leq n$}\,,\\
y'_{i,j}(0) = \overline{a_{i,j}} \text{ for $1\leq j\leq n$}\,,
\end{cases}
\end{equation}
where $y_{i,j}$ is the $j$th component of $\vy_i$ and the overline represents complex conjugation. Observe that each of these $n$ solutions satisfies $\Gamma$ boundary conditions at the origin. We denote the set of all $\vy_i$ as $\mathcal{Y}$.

Similarly, we can define another $n$ solutions by using the columns of $\begin{bmatrix}
-\beta_2^* \\ \beta_1^*
\end{bmatrix}$ as initial conditions for homogeneous problems to generate solutions $\vz_{i}$ which satisfy:
\begin{equation}\label{eq:12}
\begin{cases}
(H^\Omega-\lambda)\vz_i=\Vec{0}\, ,\\ 
\vz_i(\Vec{\ell}) = -\overline{h}_i\Vec{e}_i\, ,\\
\vz_i'(\Vec{\ell}) = \overline{g_{i}}\Vec{e}_i\, .
\end{cases}
\end{equation}
First, note that these solutions $\vz_{i}$ satisfy $\Gamma$ boundary conditions at the outer endpoints of the quantum graph. Also, we have that $z_{i,j} = 0$ for all $j\neq i$ due to the zero conditions. We denote the set of all $\vz_i$ as $\mathcal{Z}$.

Due to the rank conditions of the $\alpha$ and $\beta$ matrices imposed in equation \eqref{eq:3}, we have that $\mathcal{Y}$ is a linearly independent set and $\mathcal{Z}$ is a linearly independent set. Additionally, it can bee seen that the combined set $\mathcal{Y}\cup \mathcal{Z}$ is linearly dependent if and only if $\lambda\in \sigma(H^{\Omega}_{\Gamma})$. Thus, since $\lambda\in \rho(L_\Gamma)$, we must have that $\mathcal{Y}\cup \mathcal{Z}$ forms a linearly independent set of $2n$ functions.
\begin{definition}\label{def:columnorder}
Let $Y$ and $Z$ be $n\times n$ matrices such that the $i$th column of $Y$ is $\Vec{y}_i$ and the $i$th column of $Z$ is $\Vec{z}_i$. Observe that the ordering of these columns implies that
\begin{equation}\label{eq:151}
\begin{bmatrix}
Y(\Vec{0},\lambda)\\
Y'(\Vec{0},\lambda)
\end{bmatrix} = \begin{bmatrix}
-\alpha_2^*\\
\alpha_1^*
\end{bmatrix} \hspace{.5cm}\text{and}\hspace{.5cm} \begin{bmatrix}
Z(\Vec{\ell},\lambda)\\
Z'(\Vec{\ell},\lambda)
\end{bmatrix} = \begin{bmatrix}
-\beta_2^*\\
\beta_1^*
\end{bmatrix}.
\end{equation}
\end{definition}

\begin{definition}\label{def:evansfunction}
In this paper, we will define \textit{the} Evans function for $H^\Omega_\Gamma$ (denoted $E^\Omega_\Gamma$) as the determinant of the following fundamental solution matrix:
\begin{equation}\label{eq:152}
F(\vx,\lambda) := \begin{bmatrix}
Y(\vx,\lambda) & Z(\vx,\lambda)\\
Y'(\vx,\lambda) & Z'(\vx,\lambda)
\end{bmatrix}.
\end{equation}
That is, $E^\Omega_{\Gamma}(\lambda)=|F(\vx,\lambda)|$. The Evans function only depends on $\lambda$ due to Theorem \ref{thm:Abel}.
\end{definition}

This definition also extends to the operators defined over subgraphs like $H^{\Omega_1}_{\Gamma_1}$ and $H^{\Omega_2}_{\Gamma_2}$; we simply alter the initial value problems defining $Y$ and $Z$ to incorporate $s_1$ as reflected in the new boundary conditions.

\subsection{Particular Solution}
We wish to emulate Variation of Parameters so as to obtain a particular solution from our complementary solution set. We cannot do normal Variation of Parameters on the whole set of complementary solutions since each solution has multiple spatial variables, but we can do Variation of Parameters for each edge. So, we construct a particular solution $\vy_p$ component by component. The $j$th component of the particular solution $\vy_p$ can be built by looking at all the $y_{i,j}$ and $z_{i,j}$ (we call such components subvectors), finding a linearly independent pair, and using Variation of Parameters on them to find a particular solution to the one-dimensional second order problem:
\begin{equation}\label{eq:13}
(-\frac{\partial^2}{\partial x_i^2}+V_i-\lambda)u = v_i\hspace{.75in} x_i\in (0,\ell_i)\, .
\end{equation}

\begin{theorem}\label{thm:independent_subvectors}
Fix $\lambda\in \rho(H^{\Omega}_\Gamma)$. Consider the sets of functions $\mathcal{Y}$ and $\mathcal{Z}$ made up of solutions from equations \eqref{eq:10} and \eqref{eq:12}. For all $1\leq j\leq n$, there exists at least one index $1\leq \tau_i\leq n$ such that $y_{\tau_j,j}$ is linearly independent from $z_{j,j}$.
\end{theorem}
\begin{proof}
Suppose for contradiction we could not find another linearly independent function for some index $j$. Then there exist constants $\chi_i$ such that $y_{i,j} = \chi_iz_{j,j}$ for all $i$. Since $\mathcal{Y}\cup \mathcal{Z}$ is linearly independent, the set $S=\mathcal{Z}\cup\{\vy_i-\chi_i\vz_j: \vy_i\in \mathcal{Y}\}$ is also a linearly independent set. We can construct a new set of functions $S'$ by removing the $j$th component of every function in $S\setminus \{\vz_j\}$. Since the $j$th component for every vector in $S\setminus \{\vz_j\}$ is $0$, we lose no information by deleting them, so $S'$ is a linearly independent set of size $2n-1$ containing vectors with $n-1$ components. But this cannot be possible, since the solutions in $S'$ are linearly independent solutions to an $(n-1)$-dimensional second order ODE, of which there should be at most $2n-2$. Thus, there must be some $y_{\tau_j,j}$ meeting our desired specifications.
\end{proof}
This proof shows that our componentwise construction of $\vy_p$ is well defined for all $\lambda\in\rho(H^\Omega_\Gamma)$. We also see in the proof that the choice of each $y_{\tau_j,j}$ is local with respect to $\lambda$. Thus, for specific $\lambda$ we can use Variation of Parameters to define the $i$th entry of $\vy_p$ as
\begin{align}\label{eq:particular}
(\vy_p)_i(x_i,\lambda) :=& -
\frac{1}{D_i(\lambda)}(y_{\tau_i, i}(x_i,\lambda) \int_{x_i}^{\ell_i} v_i(t_i)z_{i,i}(t_i,\lambda)\,dt_i 
\nonumber\\&+ z_{i, i}(x_i,\lambda)\int_0^{x_i}v_i(t_i)y_{\tau_i,i}(t_i,\lambda) \,dt_i)
 \, ,
\end{align}
where $D_i:=W_0(y_{\tau_i,i},z_{i,i})$ is a Wronskian for solutions on the $i$th wire.

\subsection{Constant Tuning}
Let $S_0 = \begin{bmatrix}
Y & Z
\end{bmatrix}$. Then our general solution (which is $R_\lambda^\Gamma\vv)$ will be of the form
\begin{equation}\label{eq:14}
\vu = S_0\vc + \vy_p \, ,
\end{equation}
where $\vc$ is a vector of scalars in $\CC^{2n}$. To make sure $\vu$ satisfies the $\Gamma$ boundary conditions, we will need to adjust the components of $\vc$ so that
\begin{equation}\label{eq:15}
\gamma_\Gamma (S_0\vc) = -\gamma_\Gamma \vy_p \, .
\end{equation}

By the linearity of $\gamma_\Gamma$, we can modify equation \eqref{eq:15} by taking the trace of every column of $S_0$, arranging these in a matrix $\Tilde{S}_0$, and then multiplying this matrix by $\vc$. We define $\Tilde{S}_0$ as
\begin{align}\label{eq:16}
\Tilde{S}_0(\lambda) &= \begin{bmatrix}
\beta_1 S_0(\Vec{\ell},\lambda) + \beta_2S_0'(\Vec{\ell},\lambda)
\\
\alpha_1 S_0(\Vec{0},\lambda) + \alpha_2 S_0'(\Vec{0},\lambda)
\end{bmatrix} 
\nonumber\\&= \begin{bmatrix}
\beta_1 Y(\Vec{\ell},\lambda) + \beta_2 Y'(\Vec{\ell},\lambda) &\beta_1 Z(\Vec{\ell},\lambda) + \beta_2 Z'(\Vec{\ell},\lambda)
\\
\alpha_1 Y(\Vec{0},\lambda) + \alpha_2 Y'(\Vec{0},\lambda) & 
\alpha_1 Z(\Vec{0},\lambda) + \alpha_2 Z'(\Vec{0},\lambda)
\end{bmatrix} \nonumber \\
&= \begin{bmatrix}
\beta_1 Y(\Vec{\ell},\lambda) + \beta_2 Y'(\Vec{\ell},\lambda) & 0 
\\
0 & 
\alpha_1 Z(\Vec{0},\lambda) + \alpha_2 Z'(\Vec{0},\lambda)
\end{bmatrix}\, .
\end{align}
The zero blocks in \eqref{eq:16} appear since $Y$ and $Z$ are defined to satisfy $\alpha$ and $\beta$ conditions, respectively. We can rewrite equation $\eqref{eq:15}$ as
\begin{equation}\label{eq:17}
\Tilde{S}_0(\lambda)\vc = -\gamma_\Gamma \vy_p \, ,
\end{equation}
Due to the integral bounds we chose, we find that
\begin{equation}\label{eq:18}
\gamma_\Gamma \vy_p = \begin{bmatrix}
\Vec{0}\\ \alpha_1 \vy_p(\Vec{0},\lambda)+ \alpha_2 \vy_p'(\Vec{0},\lambda)
\end{bmatrix} \, .
\end{equation}

Equations \eqref{eq:17} and \eqref{eq:18} imply that we can always choose $c_1=c_2=...=c_n=0$, and thus we can determine $c_{n+i}$ for $1\leq i\leq n$ by solving the smaller system:
\begin{equation}\label{eq:19}
C(\lambda) \begin{bmatrix}
c_{n+1}\\\vdots\\c_{2n}
\end{bmatrix} = -\alpha_1 \vy_p(\Vec{0},\lambda)- \alpha_2 \vy_p'(\Vec{0},\lambda) \, ,
\end{equation}
where $C(\lambda)$ is given by:
\begin{equation}\label{eq:20}
C(\lambda) = 
\alpha_1 Z(\Vec{0},\lambda) + \alpha_2 Z'(\Vec{0},\lambda) \, .
\end{equation}
That is, $C(\lambda)$ is the lower right block of $\Tilde{S}_0$. We can find an exact formula for $C(\lambda)$. First, we give names to the various $\vz_i$ at the origin:
\begin{align}
-s_i(\lambda) &= z_{i,i}(0,\lambda)\nonumber \, ,\\
r_i(\lambda) &= z_{i,i}'(0,\lambda) \, .
\end{align}
Then from equation \eqref{eq:20}, $C(\lambda)$ is given by
\begin{align}\label{eq:22}
C(\lambda) &= \begin{bmatrix}
\alpha_1 & \alpha_2
\end{bmatrix} \begin{bmatrix}
\diag(-s_1(\lambda),...,-s_n(\lambda) )\\
\diag(r_1(\lambda),...,r_n(\lambda) )
\end{bmatrix}
\nonumber\\&= \begin{bmatrix}
\begin{vmatrix}
-s_1(\lambda) & -b_{1,1}\\
r_1(\lambda) & a_{1,1}
\end{vmatrix} &\hdots & \begin{vmatrix}
-s_n(\lambda) & -b_{1,n}\\
r_n(\lambda) & a_{1,n}
\end{vmatrix}\\
\vdots & \ddots & \vdots \\
\begin{vmatrix}
-s_1(\lambda) & -b_{n,1}\\
r_1(\lambda) & a_{n,1}
\end{vmatrix} & \hdots& \begin{vmatrix}
-s_n(\lambda) & -b_{n,n}\\
r_n(\lambda) & a_{n,n}
\end{vmatrix}
\end{bmatrix} \, .
\end{align}
If $C(\lambda)$ is invertible then we can solve for the coefficients $c_{n+1},...,c_{2n}$ by inversion in equation \eqref{eq:19}. The invertibility is guaranteed by the following Lemma.
\begin{lemma}\label{thm:C_and_fundamental}
Let $A$ and $B$ be $n\times n$ matrices, $C = \alpha_1 A + \alpha_2 B$, where $\alpha_i$ are defined in \eqref{eq:3}, and 
\[
F = \begin{bmatrix}
-\alpha_2^* & A\\
\alpha_1^* & B
\end{bmatrix}.
\]
Then,
\[
 (-1)^n|F|=|C| \, .
\]
\end{lemma}
\begin{proof}
We know from equation \eqref{eq:3} that there exists a set of $n$ linearly independent columns in the matrix $\begin{bmatrix}
\alpha_1 & \alpha_2
\end{bmatrix}$. Thus, by swapping columns between $\alpha_1$ and $\alpha_2$, we can generate a new block matrix $\begin{bmatrix}
\Tilde{\alpha}_1 & \Tilde{\alpha}_2
\end{bmatrix}$ such that $\text{rank} (\Tilde{\alpha}_2) = n$. Let $k\in \{1,2,...,n\}$ represent the number of columns that we swap between $\alpha_1$ and $\alpha_2$ to generate these new matrices (although there may be different $k$ for different column swapping configurations). Now, for every pair of columns swapped, swap the rows with the same indices in $F$ to generate a new matrix $\Hat{F}$. We immediately observe that
\begin{equation}\label{eq:56}
|\Hat{F}|=(-1)^k|F|.
\end{equation}
We label the blocks of $\Hat{F}$ as follows:
\begin{equation}\label{eq:57}
\Hat{F} = \begin{bmatrix}
\Hat{\alpha}_2 & \Hat{A}\\
\Hat{\alpha}_1 & \Hat{B}
\end{bmatrix},
\end{equation}
and observe that
\begin{equation}\label{eq:58}
|\Hat{\alpha}_2| = (-1)^{n-k}|\Tilde{\alpha}_2|,
\end{equation}
since the transformation from $-\alpha_2^*$ to $\Hat{\alpha}_2$ involves swapping away $k$ of the $n$ rows, each of which carries away the added negative from the initial scaling of $\alpha_2^*$.

Note that, since matrix multiplication is defined in terms of vector inner products, by switching rows and columns of the same index to get $\Hat{F}$ and $\begin{bmatrix}
\Tilde{\alpha}_1 & \Tilde{\alpha}_2
\end{bmatrix}$, respectively, we preserve the values of the inner products defining the entries of the matrix product between $\begin{bmatrix}
\alpha_1 & \alpha_2
\end{bmatrix}$ and $F$.

The final object we need for this proof is the invertible block triangular matrix $J$ defined as
\begin{equation}\label{eq:60}
J = \begin{bmatrix}
I_n & 0\\
\Tilde{\alpha}_1 & \Tilde{\alpha}_2
\end{bmatrix},
\end{equation}
with determinant $|J| = |\Tilde{\alpha}_2|$, which is nonzero since we specifically constructed $\Tilde{\alpha}_2$ to consist of $n$ linearly independent columns. We can calculate the product $J\hat{F}$ as
\begin{align}
\label{eq:91}
J\hat{F} = \begin{bmatrix}
I_n & 0\\
\Tilde{\alpha}_1 & \Tilde{\alpha}_2
\end{bmatrix} \begin{bmatrix}
\hat{\alpha}_2 & \hat{A}\\
\hat{\alpha}_1 & \hat{B}
\end{bmatrix} &= \begin{bmatrix}
\hat{\alpha}_2 & \hat{A}\\
\Tilde{\alpha}_1\hat{\alpha}_2+\Tilde{\alpha}_2\hat{\alpha}_1 & \Tilde{\alpha}_1\hat{A}+\Tilde{\alpha}_2\hat{B} 
\end{bmatrix}\nonumber\\
&= \begin{bmatrix}
\hat{\alpha}_2 & \hat{A}\\
-\alpha_1\alpha_2^*+\alpha_2\alpha_1^* & \alpha_1 A+\alpha_2 B 
\end{bmatrix}\nonumber\\
&\myeq{\eqref{eq:3}}\,\,\, \begin{bmatrix}
\hat{\alpha}_2 & \hat{A}\\
0 & C
\end{bmatrix}.
\end{align}
So, on the one hand, we have that
\begin{equation}\label{eq:68}
|J\Hat{F}| = |\Hat{\alpha}_2||C|\,\,\myeq{\eqref{eq:58}} \,  (-1)^{n-k}|\Tilde{\alpha}_2||C| \, ,
\end{equation}
while on the other hand, we have
\begin{equation}\label{eq:65}
|J\Hat{F}| = |J||\Hat{F}|\,\, \myeq{\eqref{eq:56}}\,\, |\Tilde{\alpha}_2||F|(-1)^k\, .
\end{equation}
Thus, by equating equations \eqref{eq:68} and \eqref{eq:65}, multiplying both sides by $(-1)^{n-k}$, and dividing by $|\Tilde{\alpha}_2|$, we obtain the equality we set out to prove.
\end{proof}
In particular, using $C(\lambda)$ as defined in \eqref{eq:20} and $F(\vx,\lambda)$ defined to be the fundamental matrix in equation \eqref{eq:152}, we can apply Lemma \ref{thm:C_and_fundamental} to obtain the following relation:
\begin{equation}\label{eq:154}
(-1)^n|C(\lambda)| = |F(\vx,\lambda)| = E^{\Omega}_\Gamma(\lambda).
\end{equation}
Hence, $|C(\lambda)|$ is zero if and only if $\lambda\in \sigma(H^{\Omega}_\Gamma)$. Since we have supposed $\lambda\in\rho(H^{\Omega}_\Gamma)$, we have that $C(\lambda)$ is invertible and we can apply Cramer's rule to solve equation \eqref{eq:19}. Of course, the right side of this equation is the bottom half of $-\gamma_\Gamma \vy_p$, which can be calculated to be
\begin{align}\label{eq:25}
&-\alpha_1 \vy_p(\Vec{0},\lambda)- \alpha_2 \vy_p'(\Vec{0},\lambda)\nonumber
\\&=\begin{bmatrix}
\sum_{i=1}^n\begin{vmatrix}
-\overline{b_{\tau_i,i}} & -b_{1, i}\\ \overline{a_{\tau_i,i}}&a_{1, i}
\end{vmatrix}\frac{\int_{0}^{\ell_i} v_i(t_i)z_{i,i}(t_i,\lambda) \,dt_i}{D_i(\lambda)}
\\
\vdots\\
\sum_{i=1}^n\begin{vmatrix}
-\overline{b_{\tau_i,i}} & -b_{n, i}\\ \overline{a_{\tau_i,i}}&a_{n, i}
\end{vmatrix}\frac{\int_{0}^{\ell_i} v_i(t_i)z_{i,i}(t_i,\lambda) \,dt_i}{D_i(\lambda)}
\end{bmatrix} \, .
\end{align}
So, applying Cramer's rule to solve equation \eqref{eq:19}, we obtain
\begin{align} \label{eq:27}
c_{n+i} = \frac{|C_i(\lambda,\vv)|}{|C(\lambda)|} \, ,
\end{align}
where $C_i$ is defined as $C$ with the $i$th column replaced by the right side of equation \eqref{eq:25}. It is this replaced column that introduces $\vv$ dependence which will be vital later, so we make special note of it.

Thus, the resolvent is given by:
\begin{equation}\label{eq:resolvent_formula}
(R_\lambda^\Gamma\vv)(\vx) = \left ( \sum_{i=1}^n c_{n+i}(\lambda,\vv) \vz_i(\vx,\lambda) \right ) +\vy_p(\vx,\lambda,\vv) \, .
\end{equation}

As previously mentioned, the choice of the $n$ different $D_i$ matrices depends on $\lambda$ locally. However, our construction method overall will actually give us the resolvent globally, for all $\lambda\in \rho(H^{\Omega}_\Gamma)$. Since our constructed $R_\lambda^\Gamma\vv$ is a meromorphic function which is identical to the resolvent on some disc over which our choice of $D_i$ matrices is valid, it is in fact equal to the resolvent globally due to the unique analytic continuation of holomorphic functions (which can be adjusted to extend to meromorphic functions).

\section{From Resolvents to Maps}
Let $\lambda\in \rho(H^{\Omega}_\Gamma)$, and fix $f\in \CC^{2n}$. If $\vu_{\Gamma}\in H^2(\Omega)$ solves
\begin{equation}\label{eq:28}
\begin{cases}
(H^\Omega-\lambda)\vu_\Gamma = \Vec{0} \, ,\\
\gamma_\Gamma \vu = \vf \, , 
\end{cases}
\end{equation}
then, as seen in \cite{gesztesy2008generalized}, one can use integration by parts to show that
\begin{align}\label{eq:ibp}
(\vu_\Gamma, \vv) &= (\vu_\Gamma, (H^\Omega_\Gamma-\bar{\lambda} )R_{\bar{\lambda}}^\Gamma\vv)
\nonumber\\&= ((H^\Omega_\Gamma-\bar{\lambda} )\vu_\Gamma, R_{\bar{\lambda}}^\Gamma\vv)+ \langle \gamma_N\vu_\Gamma, \gamma_D R_{\bar{\lambda}}^\Gamma\vv \ra - \langle \gamma_D \vu_\Gamma, \gamma_N R_{\bar{\lambda}}^\Gamma\vv \ra
\nonumber\\&= \langle \gamma_N\vu_\Gamma, \gamma_D R_{\bar{\lambda}}^\Gamma\vv \ra - \langle \gamma_D \vu_\Gamma, \gamma_N R_{\bar{\lambda}}^\Gamma\vv \ra \,\nonumber
\\&= \langle \gamma_N\vu_\Gamma, \gamma_D \overline{R_{\lambda}^\Gamma\vv} \ra - \langle \gamma_D \vu_\Gamma, \gamma_N \overline{R_{\lambda}^\Gamma\vv} \ra \, ,
\end{align}
where $(\cdot,\cdot)$ is the $L^2$ inner product,$\la \cdot, \cdot \ra$ is the $\CC^{2n}$ inner product, and $\vv\in L^2(\Omega)$. Both are linear in their first component and antilinear in their second. We wish to somehow incorporate $\vf$ in this formula. Taking inspiration from \cite{qgraphIntro}, we introduce several objects. They can be constructed based on $\alpha$ or $\beta$ matrices, but we can denote the construction of both at once in terms of the block diagonal matrices:
\begin{equation}\label{eq:delta_def}
\delta_i = \begin{bmatrix}
\beta_i & 0\\
0 & \alpha_i
\end{bmatrix} \hspace{.25in} i\in \{1,2\} \, .
\end{equation}
First, we define a function $\eta$ by:
\begin{equation}\label{eq:eta_def}
\eta(k) = -(\delta_1 + ik\delta_2)^{-1}(\delta_1-ik\delta_2)\, .
\end{equation}
We use this function to define the unitary matrix $\mathcal{U}$:
\begin{equation}\label{eq:u_def}
\mathcal{U} = \eta(-1) = -(\delta_1 - i\delta_2)^{-1}(\delta_1+i\delta_2)\, .
\end{equation}
We will now modify the following basic equation:
\begin{equation}\label{eq:29}
\delta_1 \begin{bmatrix}
\vu_\Gamma(\Vec{\ell},\lambda)\\
\vu_\Gamma(\Vec{0},\lambda)
\end{bmatrix} + \delta_2 \begin{bmatrix}
\vu_\Gamma'(\Vec{\ell},\lambda)\\
\vu_\Gamma'(\Vec{0},\lambda)
\end{bmatrix} = \vf\, .
\end{equation}
Multiplying equation \eqref{eq:29} by $-2i(\delta_1-i\delta_2)^{-1}$ from the left gives us:
\begin{equation}\label{eq:31}
i(\mathcal{U}-I_{2n})\begin{bmatrix}
\vu_\Gamma(\Vec{\ell},\lambda)\\
\vu_\Gamma(\Vec{0},\lambda)
\end{bmatrix} + (\mathcal{U}+I_{2n})\begin{bmatrix}
\vu_\Gamma'(\Vec{\ell},\lambda)\\
\vu_\Gamma'(\Vec{0},\lambda)
\end{bmatrix} = -2i(\delta_1-i\delta_2)^{-1}\vf \, .
\end{equation}
We can get more useful identities from this one by introducing the Dirichlet, Neumann, and Robin projections.

\begin{definition}\label{def:projections}
The Dirichlet Projection $P_D$ is the projection onto the kernel of $\delta_2$. It can alternatively be defined as the projection onto the eigenspace of $\mathcal{U}$ corresponding to the eigenvalue $-1$.

The Neumann Projection $P_N$ is the projection onto the kernel of $\delta_1$. It can alternatively be defined as the projection onto the eigenspace of $\mathcal{U}$ corresponding to the eigenvalue $1$.

The Robin Projection $P_R$ is defined in terms of the other two by the equation
\begin{equation}\label{eq:robin_proj}
P_D+P_N+P_R = I_{2n} \, .
\end{equation}
It is also the projection onto the eigenspace of $\mathcal{U}$ corresponding to all other eigenvalues not $\pm 1$.
\end{definition}

\begin{remark}
All three projectors commute with $\mathcal{U}$, $(\mathcal{U}+I_{2n})$, and $(\mathcal{U}-I_{2n})$. We also get from the definitions that:
\begin{align}\label{eq:32}
(\mathcal{U} + I_{2n})P_D &= 0_{2n}\nonumber\, ,\\
(\mathcal{U} - I_{2n})P_D &= -2P_D\nonumber\, ,\\
(\mathcal{U} + I_{2n})P_N &= 2P_N\nonumber\, ,\\
(\mathcal{U} - I_{2n})P_N &= 0_{2n}\, .
\end{align}
\end{remark}

One more important object is $\Lambda:\text{ran}\, P_R\rightarrow \text{ran}\, P_R$. It is an invertible self-adjoint operator constructed as follows:
\begin{equation}\label{eq:30}
\Lambda := -i(\mathcal{U}+I_{2n})^{-1}_R(\mathcal{U}-I_{2n}) \, ,
\end{equation}
where $(\mathcal{U}+I_{2n})_R$ is the restriction of $(\mathcal{U}+I_{2n})$ to the range of $P_R$ (this restriction has domain and range $\ran P_R$, and is invertible). Since $\delta_1$ and $\delta_2$ are both block diagonal, these new objects are all also block diagonal as well.

With the introduction of $\Lambda$, we can now list one more set of defining properties of these objects. Since $R^\Gamma_\lambda$ satisfies $\Gamma$ boundary conditions, we have that:
\begin{align}\label{eq:projection_equations}
P_D \begin{bmatrix}
(R^\Gamma_\lambda\vv)(\Vec{\ell},\lambda)\\
(R^\Gamma_\lambda\vv)(\Vec{0},\lambda)
\end{bmatrix} &= \Vec{0}
\nonumber \, ,\\
P_N \begin{bmatrix}
(R^\Gamma_\lambda\vv)'(\Vec{\ell},\lambda)\\
(R^\Gamma_\lambda\vv)'(\Vec{0},\lambda)
\end{bmatrix} &= \Vec{0}
\nonumber\, ,\\
P_R\begin{bmatrix}
(R^\Gamma_\lambda\vv)'(\Vec{\ell},\lambda)\\
(R^\Gamma_\lambda\vv)'(\Vec{0},\lambda)
\end{bmatrix} &= \Lambda P_R \begin{bmatrix}
(R^\Gamma_\lambda\vv)(\Vec{\ell},\lambda)\\
(R^\Gamma_\lambda\vv)(\Vec{0},\lambda)
\end{bmatrix}\, .
\end{align}

Here our derivations diverge from those in \cite{qgraphIntro}, where equations \eqref{eq:33}, \eqref{eq:34}, and \eqref{eq:35} are found with $\vf=\Vec{0}$. Multiplying equation $\eqref{eq:31}$ by $P_D$ from the left, we obtain:
\begin{equation}\label{eq:33}
P_D \begin{bmatrix}
\vu_\Gamma(\Vec{\ell},\lambda)\\
\vu_\Gamma(\Vec{0},\lambda)
\end{bmatrix} = P_D(\delta_1-i\delta_2)^{-1}\vf \,.
\end{equation}

Multiplying equation \eqref{eq:31} by $P_N$ from the left yields:
\begin{equation}\label{eq:34}
P_N \begin{bmatrix}
\vu_\Gamma'(\Vec{\ell},\lambda)\\
\vu_\Gamma'(\Vec{0},\lambda)
\end{bmatrix} = -iP_N(\delta_1-i\delta_2)^{-1}\vf \,.
\end{equation}

Multiplying equation \eqref{eq:31} by $P_R$ from the left gives:
\begin{align}\label{eq:35}
&i(\U - I_{2n})P_R\begin{bmatrix}
\vu_\Gamma(\Vec{\ell},\lambda)\\
\vu_\Gamma(\Vec{0},\lambda)
\end{bmatrix} + (\U + I_{2n})P_R\begin{bmatrix}
\vu_\Gamma'(\Vec{\ell},\lambda)\\
\vu_\Gamma'(\Vec{0},\lambda)
\end{bmatrix} = -2iP_R(\delta_1 - i\delta_2)^{-1}\vf\nonumber
\\
&\Rightarrow(\U + I_{2n})P_R\begin{bmatrix}
\vu_\Gamma'(\Vec{\ell},\lambda)\\
\vu_\Gamma'(\Vec{0},\lambda)
\end{bmatrix} = -i(\U - I_{2n})P_R\begin{bmatrix}
\vu_\Gamma(\Vec{\ell},\lambda)\\
\vu_\Gamma(\Vec{0},\lambda)
\end{bmatrix} \nonumber
\\&\hspace{1.8in}-2iP_R(\delta_1 - i\delta_2)^{-1}\vf\nonumber
\\&\Rightarrow P_R\begin{bmatrix}
\vu_\Gamma'(\Vec{\ell},\lambda)\\
\vu_\Gamma'(\Vec{0},\lambda)
\end{bmatrix} = -i(\U + I_{2n})^{-1}_R(\U - I_{2n})P_R\begin{bmatrix}
\vu_\Gamma(\Vec{\ell},\lambda)\\
\vu_\Gamma(\Vec{0},\lambda)
\end{bmatrix} \nonumber
\\&\hspace{1.2in}-2i(\U + I_{2n})_R^{-1}P_R(\delta_1 - i\delta_2)^{-1}\vf\nonumber
\\&\Rightarrow P_R\begin{bmatrix}
\vu_\Gamma'(\Vec{\ell},\lambda)\\
\vu_\Gamma'(\Vec{0},\lambda)
\end{bmatrix} = \Lambda P_R\begin{bmatrix}
\vu_\Gamma(\Vec{\ell},\lambda)\\
\vu_\Gamma(\Vec{0},\lambda)
\end{bmatrix} -2i(\U + I_{2n})_R^{-1}P_R(\delta_1 - i\delta_2)^{-1}\vf\,.
\end{align}

We can rewrite the previous three equations in terms of the Dirichlet and Neumann traces if we account for the normal vector scaling in Neumann traces. So, given some $2n\times 2n$ matrix $M$, we introduce the new matrix $\Tilde{M}$ defined as
\begin{equation}\label{eq:36}
\Tilde{M} = M\begin{bmatrix}
I_n & 0\\
0 & -I_n
\end{bmatrix}\,.
\end{equation}
Then we may rewrite equations \eqref{eq:33}, \eqref{eq:34}, \eqref{eq:35}, and the last line of \eqref{eq:34} as
\begin{align}\label{eq:37}
P_D \gamma_D\vu_\Gamma &= P_D(\delta_1-i\delta_2)^{-1}\vf \,,
\nonumber\\
P_N \gamma_N\vu_\Gamma &= -i\Tilde{P}_N(\delta_1-i\delta_2)^{-1}\vf\,,
\nonumber \\
P_R \gamma_N\vu_\Gamma &= \Lambda \Tilde{P}_R\gamma_D\vu_\Gamma-2i(\U + I_{2n})^{-1}_R \Tilde{P}_R(\delta_1 -  i\delta_2)^{-1}\vf\,,
\nonumber \\
P_R \gamma_N R_\lambda^{\Gamma}\vv &= \Lambda \Tilde{P}_R\gamma_DR_\lambda^{\Gamma}\vv\,.
\end{align}
Note that, due to the block diagonal nature of all of the new matrices involved, we can safely rewrite the domain and range of $\Lambda$ as $\text{ran}\,\Tilde{P}_R$. Using all this machinery, we can alter our inner product formula.
\begin{theorem}\label{thm:inner_product}
Let $\lambda\in \rho(H^{\Omega}_\Gamma)$, $\vf\in \CC^{2n}$, and $\vv\in L^2(\Omega)$. Let $Q$ be the image of $\ran \Tilde{P}_R$ under $\delta_2$. If $\vu_\Gamma\in H^2(\Omega)$ solves equation \eqref{eq:28}, then
\begin{align} \label{eq:the_inner_product}
(\vu_\Gamma,\vv) =&\langle -i\Tilde{P}_N(\delta_1 -  i\delta_2)^{-1}\vf, \gamma_D \overline{R_{\lambda}^\Gamma \vv}\rangle \nonumber
\\&+\langle (\delta_2)_Q^{-1}(\delta_1 - i\delta_2)_R \Tilde{P}_R(\delta_1 -  i\delta_2)^{-1}\vf, \gamma_D \overline{R_{\lambda}^\Gamma \vv}\rangle \nonumber
\\&-\langle P_D(\delta_1-i\delta_2)^{-1}\vf, \gamma_N \overline{R_{\lambda}^\Gamma \vv}\rangle \,.
\end{align}
\end{theorem}
\begin{proof}
By utilizing our new self-adjoint projections, we find that:
\begin{align}\label{eq:big_inner_prod}
(\vu_{\Gamma},\vv) \myeq{\eqref{eq:ibp}}\,\,&\langle\gamma_N \vu_\Gamma, \gamma_D \overline{R_{\lambda}^\Gamma \vv}\rangle - \langle \gamma_D \vu_\Gamma, \gamma_N \overline{R_{\lambda}^\Gamma \vv}\rangle
\nonumber\\
\myeq{\eqref{eq:robin_proj}}\,\,&
\langle\gamma_N \vu_\Gamma, (P_D+P_N+P_R)\gamma_D \overline{R_{\lambda}^\Gamma \vv}\rangle \nonumber
\\&- \langle \gamma_D \vu_\Gamma, (P_D+P_N+P_R)\gamma_N \overline{R_{\lambda}^\Gamma \vv}\rangle
\nonumber\\
\myeq{\eqref{eq:projection_equations}}\,\,&\langle P_N\gamma_N \vu_\Gamma, \gamma_D \overline{R_{\lambda}^\Gamma \vv}\rangle +\langle P_R\gamma_N \vu_\Gamma, \gamma_D \overline{R_{\lambda}^\Gamma \vv}\rangle 
\nonumber\\
&-\langle P_D\gamma_D \vu_\Gamma, \gamma_N \overline{R_{\lambda}^\Gamma \vv}\rangle - \langle \gamma_D \vu_\Gamma, P_R\gamma_N \overline{R_{\lambda}^\Gamma \vv}\rangle
\nonumber\\
\myeq{\eqref{eq:37}}\,\,&\langle -i\Tilde{P}_N(\delta_1 - i\delta_2)^{-1}\vf, \gamma_D \overline{R_{\lambda}^\Gamma \vv}\rangle 
\nonumber\\
&+\langle \Lambda \Tilde{P}_R\gamma_D\vu_\Gamma-2i(\U + I_{2n})^{-1}_R \Tilde{P}_R(\delta_1 -  i\delta_2)^{-1}\vf, \gamma_D \overline{R_{\lambda}^\Gamma \vv}\rangle
\nonumber\\
&-\langle P_D(\delta_1-i\delta_2)^{-1}\vf, \gamma_N \overline{R_{\lambda}^\Gamma \vv}\rangle - \langle \gamma_D \vu_\Gamma, \Lambda \Tilde{P}_R\gamma_D \overline{R_{\lambda}^\Gamma \vv}\rangle
\nonumber\\
=&\langle -i\Tilde{P}_N(\delta_1 - i\delta_2)^{-1}\vf, \gamma_D \overline{R_{\lambda}^\Gamma \vv}\rangle 
\nonumber\\
&+\langle \Tilde{P}_R\Lambda \Tilde{P}_R\gamma_D\vu_\Gamma-2i(\U + I_{2n})^{-1}_R \Tilde{P}_R(\delta_1 -  i\delta_2)^{-1}\vf, \gamma_D \overline{R_{\lambda}^\Gamma \vv}\rangle
\nonumber\\
&-\langle P_D(\delta_1-i\delta_2)^{-1}\vf, \gamma_N \overline{R_{\lambda}^\Gamma \vv}\rangle - \langle \gamma_D \vu_\Gamma, \Tilde{P}_R\Lambda \Tilde{P}_R\gamma_D \overline{R_{\lambda}^\Gamma \vv}\rangle
\nonumber\\
=&\langle -i\Tilde{P}_N(\delta_1 -  i\delta_2)^{-1}\vf, \gamma_D \overline{R_{\lambda}^\Gamma \vv}\rangle 
\nonumber\\
&+\langle -2i(\U + I_{2n})^{-1}_R \Tilde{P}_R(\delta_1 -  i\delta_2)^{-1}\vf, \gamma_D \overline{R_{\lambda}^\Gamma \vv}\rangle
\nonumber\\
&-\langle P_D(\delta_1-i\delta_2)^{-1}\vf, \gamma_N \overline{R_{\lambda}^\Gamma \vv}\rangle
\nonumber\\
&+\langle [\Lambda -\Lambda^*]\Tilde{P}_R\gamma_D\vu_\Gamma, \Tilde{P}_R\gamma_D \overline{R_{\lambda}^\Gamma \vv}\rangle\,.
\end{align}
As the Cayley Transform of the unitary matrix $\mathcal{U}$, $\Lambda$ is self-adjoint as a matrix from $\ran \Tilde{P}_R$ to $\ran \Tilde{P}_R$. Thus, the last term in the last equality of equation \eqref{eq:big_inner_prod} is $0$.

As an additional simplification, we wish to eliminate $\U$ from the term $-2i(\U + I)^{-1}_R \Tilde{P}_R(\delta_1 -  i\delta_2)^{-1}\vf$. We observe that $\U+I_{2n}$ can be rewritten as
\[
\U+I_{2n} = -2i(\delta_1 - i\delta_2)^{-1}\delta_2 \,.
\]
Since $(\U+I_{2n})_R:\ran \Tilde{P}_R\rightarrow \ran\Tilde{P}_R$ is invertible, we have that if the image of $\ran \Tilde{P}_R$ under $\delta_2$ is some $Q\subseteq \CC^{2n}$, then the image of $Q$ under $(\delta_1 - i\delta_2)^{-1}$ must be $\ran \Tilde{P}_R$. Thus, $(\delta_1-i\delta_2)_R$, the restriction of $(\delta_1-i\delta_2)$ to $\ran \Tilde{P}_R$ will map from $\ran \Tilde{P}_R$ to $Q$, and $(\delta_2)_Q^{-1}$ is defined as a map from $Q$ to $\ran \Tilde{P}_R$, so we see that
\[
-2i(\U + I_{2n})^{-1} \Tilde{P}_R(\delta_1 -  i\delta_2)^{-1}\vf = (\delta_2)_Q^{-1}(\delta_1 - i\delta_2)_R \Tilde{P}_R(\delta_1 -  i\delta_2)^{-1}\vf \,.
\]
With these simplifications to equation \eqref{eq:big_inner_prod}, we obtain the desired formula.
\end{proof}

Theorem \ref{thm:inner_product} can be used to construct $\vu_\Gamma$ with an arbitrary trace value $\vf$, by using $R_{\lambda}^\Gamma\vv$. In particular, by setting $\vv$ to constant functions taking the form of the various standard basis vectors of $\RR^n$, we can solve for the components of $\vu_{\Gamma,i}$ one by one. Part of this process will involve the integrands produced by two different inner products. Note that the various $C_i(\lambda,\Vec{e}_j)$ include a column (from equation \eqref{eq:25}) in which all entries carry definite integrals of $\frac{z_{j,j}}{D_j}$ from $0$ to $\ell_j$ with respect to the placeholder $t_j$. Thus, taking a determinant and expanding down this column, we see that every term in this sum defining $|C_i(\lambda,\Vec{e}_j)|$ still includes a factor carrying this integral. By absorbing all constants and combining integral terms, we can write the whole determinant as one definite integral from $0$ to $\ell_j$. If we define $\mathcal{C}_i(\lambda,\Vec{e}_j)$ as $C_i(\lambda,\Vec{e}_j)$ without these integrals (that is, leaving behind a column of only the scaling constants), then we easily see that 
\begin{equation}\label{Ck}
    |C_i(\lambda,\Vec{e}_j)| = \int_0^{\ell_j}|\mathcal{C}_i(\lambda,\Vec{e}_j)|\frac{z_{j,j}(t_j,\lambda)}{D_j(\lambda)} \,dt_j.
\end{equation}

Before we proceed, we define the following adjustment vectors:
\begin{align}\label{eq:38}
\Vec{L}_i &=
-i\Tilde{P}_N(\delta_1 -  i\delta_2)^{-1}\Vec{e}_i \,,
\nonumber\\ \Vec{M}_i &=(\delta_2)_Q^{-1}(\delta_1 - i\delta_2)_R \Tilde{P}_R(\delta_1 -  i\delta_2)^{-1}\Vec{e}_i \,,
\nonumber\\\Vec{N}_i &= -P_D(\delta_1-i\delta_2)^{-1}\Vec{e}_i \,.
\end{align}

\begin{theorem}\label{thm:general_ugamma}
Let $\lambda\in \rho(H^{\Omega}_\Gamma)$, and let $\vu_{\Gamma,i}\in H^2(\Omega)$ be the unique solution to
\[
\begin{cases}
(H^{\Omega}-\lambda)\vu_{\Gamma,i}=\Vec{0}\\
\gamma_{\Gamma}\vu_{\Gamma,i}=\Vec{e}_i.
\end{cases}
\] 
Then,
the $j$th component $(\vu_{\Gamma,i})_j$ can be calculated using the formula:
\begin{align}\label{eq:general_ugamma}
(\vu_{\Gamma,i})_j(x_j,\lambda)=&\la \Vec{L}_i + \Vec{M}_i, \begin{bmatrix}
\frac{z_{j,j}(x_j,\lambda)}{D_j(\lambda)}\frac{|\mathcal{C}_1(\lambda,\Vec{e}_j)|}{|C(\lambda)|}z_{1,1}(\ell_1,\lambda)
\\
\vdots
\\
\frac{z_{j,j}(\ell_j,\lambda)(z_{j,j}(x_j,\lambda)|\mathcal{C}_j(\lambda,\Vec{e}_j)|-|C(\lambda)|y_{\tau_j, j}(x_j,\lambda))}{|C(\lambda)|D_j(\lambda)}
\\
\vdots
\\  \frac{z_{j,j}(x_j,\lambda)}{D_j(\lambda)}\frac{|\mathcal{C}_n(\lambda,\Vec{e}_j)|}{|C(\lambda)|}z_{n,n}(\ell_n,\lambda)
\\ \frac{z_{j,j}(x_j,\lambda)}{D_j(\lambda)}\frac{|\mathcal{C}_1(\lambda,\Vec{e}_j)|}{|C(\lambda)|}z_{1,1}(0,\lambda)
\\
\vdots
\\
\frac{z_{j,j}(x_j,\lambda)(z_{j,j}(0,\lambda)|\mathcal{C}_j(\lambda,\Vec{e}_j)|-|C(\lambda)|y_{\tau_j, j}(0,\lambda))}{|C(\lambda)|D_j(\lambda)}
\\
\vdots
\\
\frac{z_{j,j}(x_j,\lambda)}{D_j(\lambda)}\frac{|\mathcal{C}_n(\lambda,\Vec{e}_j)|}{|C(\lambda)|}z_{n,n}(0,\lambda)
\end{bmatrix} \ra 
\nonumber\\ &+ \la \Vec{N}_i, \begin{bmatrix}
\frac{z_{j,j}(x_j,\lambda)}{D_j(\lambda)}\frac{|\mathcal{C}_1(\lambda,\Vec{e}_j)|}{|C(\lambda)|}z_{1,1}'(\ell_1,\lambda)
\\
\vdots
\\
\frac{z_{j,j}'(\ell_j,\lambda)(z_{j,j}(x_j,\lambda)|\mathcal{C}_j(\lambda,\Vec{e}_j)|-|C(\lambda)|y_{\tau_j, j}(x_j,\lambda))}{|C(\lambda)|D_j(\lambda)}
\\
\vdots
\\  \frac{z_{j,j}(x_j,\lambda)}{D_j(\lambda)}\frac{|\mathcal{C}_n(\lambda,\Vec{e}_j)|}{|C(\lambda)|}z_{n,n}'(\ell_n,\lambda)
\\ -\frac{z_{j,j}(x_j,\lambda)}{D_j(\lambda)}\frac{|\mathcal{C}_1(\lambda,\Vec{e}_j)|}{|C(\lambda)|}z_{1,1}'(0,\lambda)
\\
\vdots
\\
-\frac{z_{j,j}(x_j,\lambda)(z_{j,j}'(0,\lambda)|\mathcal{C}_j(\lambda,\Vec{e}_j)|-|C(\lambda)|y_{\tau_j, j}'(0,\lambda))}{|C(\lambda)|D_j(\lambda)}
\\
\vdots
\\
-\frac{z_{j,j}(x_j,\lambda)}{D_j(\lambda)}\frac{|\mathcal{C}_n(\lambda,\Vec{e}_j)|}{|C(\lambda)|}z_{n,n}'(0,\lambda)
\end{bmatrix}\ra \,,
\end{align}
where the terms with $|C(\lambda)|$ in the numerator occur in the $j$th and $(n+j)$th components only.
\end{theorem}
\begin{proof}
First we calculate the general Dirichlet and Neumann traces of $R_{\lambda}^\Gamma\vv$. Since each column in $Z$ is only nonzero in one component (see formula \eqref{eq:12}), by using the definition of the resolvent from equation \eqref{eq:resolvent_formula} (which iteslf references the particular solution formula from equation \eqref{eq:particular}), we find that for $k\in \{1,...,n\}$:
\[
(\gamma_D R_{\lambda}^\Gamma\vv)_k = \frac{|C_k({\lambda},\vv)|}{|C({\lambda})|}z_{k,k}(\ell_k,{\lambda}) -z_{k, k}(\ell_k,{\lambda})\int_{0}^{\ell_k}\frac{v_k(t_k)y_{\tau_k, k}(t_k,{\lambda})}{D_k({\lambda})}  \,dt_k,
\]
\[
(\gamma_D R_{\lambda}^\Gamma\vv)_{n+k} = \frac{|C_k(\lambda,\vv)|}{|C(\lambda)|}z_{k,k}(0,\lambda) -y_{\tau_k, k}(0,\lambda)\int_{0}^{\ell_k} \frac{v_k(t_k)z_{k,k}(t_k,\lambda)}{D_k(\lambda)} \,dt_k,
\]
\[
(\gamma_N R_{\lambda}^\Gamma\vv)_k = \frac{|C_k(\lambda,\vv)|}{|C(\lambda)|}z_{k,k}'(\ell_k,\lambda) -z_{k, k}'(\ell_k,\lambda)\int_{0}^{\ell_k}\frac{v_k(t_k)y_{\tau_k, k}(t_k,\lambda)}{D_k(\lambda)} \,dt_k,
\]
\[
(\gamma_N R_{\lambda}^\Gamma\vv)_{n+k} =-(\frac{|C_k(\lambda,\vv)|}{|C(\lambda)|}z_{k,k}'(0,\lambda) -y_{\tau_k, k}'(0,\lambda)\int_{0}^{\ell_k} \frac{v_k(t_k)z_{k,k}(t_k,\lambda)}{D_k(\lambda)} \,dt_k).
\]
When $\vv=\Vec{e}_j$, the components of these traces can be simplified as follows:
\[
(\gamma_D R_{\lambda}^\Gamma\Vec{e}_j)_k = \begin{cases}
\frac{|C_k({\lambda},\Vec{e}_j)|}{|C({\lambda})|}z_{k,k}(\ell_k,{\lambda}), &k\neq j\\
\frac{|C_j({\lambda},\Vec{e}_j)|}{|C({\lambda})|}z_{j,j}(\ell_j,{\lambda}) -z_{j, j}(\ell_j,{\lambda})\int_{0}^{\ell_j}\frac{y_{\tau_j, j}(t_j,{\lambda})}{D_j({\lambda})} \,dt_j, & k=j
\end{cases}
\]
\[
(\gamma_D R_{\lambda}^\Gamma\Vec{e}_j)_{n+k} =\begin{cases}
\frac{|C_k({\lambda},\Vec{e}_j)|}{|C({\lambda})|}z_{k,k}(0,{\lambda}), & k\neq j\\
\frac{|C_j({\lambda},\Vec{e}_j)|}{|C({\lambda})|}z_{j,j}(0,{\lambda}) -y_{\tau_j, j}(0,{\lambda})\int_{0}^{\ell_j} \frac{z_{j,j}(t_j,{\lambda})}{D_j({\lambda})} \,dt_j, & k=j
\end{cases}
\]
\[
(\gamma_N R_\lambda^\Gamma\Vec{e}_j)_k = \begin{cases}
\frac{|C_k(\lambda,\Vec{e}_j)|}{|C(\lambda)|}z_{k,k}'(\ell_k,\lambda), & k\neq j \\
\frac{|C_j(\lambda,\Vec{e}_j)|}{|C(\lambda)|}z_{j,j}'(\ell_j,\lambda) -z_{j, j}'(\ell_j,\lambda)\int_{0}^{\ell_j}\frac{y_{\tau_j, j}(t_j,\lambda)}{D_j(\lambda)} \,dt_j, & k=j
\end{cases}
\]
\[
(\gamma_N R_\lambda^\Gamma\Vec{e}_j)_{n+k} = \begin{cases}
-\frac{|C_k(\lambda,\Vec{e}_j)|}{|C(\lambda)|}z_{k,k}'(0,\lambda), & k\neq j\\
-(\frac{|C_j(\lambda,\Vec{e}_j)|}{|C(\lambda)|}z_{j,j}'(0,\lambda) -y_{\tau_j, j}'(0,\lambda)\int_{0}^{\ell_j} \frac{z_{j,j}(t_j,\lambda)}{D_j(\lambda)} \,dt_j), & k=j
\end{cases}
\]
Thus, we can recover the $j$th component of $\vu_{\Gamma,i}$ by setting $\vv=\Vec{e}_j$, equating the integrands on both sides of equation \eqref{eq:the_inner_product} and using formulas \eqref{Ck}, \eqref{eq:38}, completing the proof.
\end{proof}

\section{Eigenvalue Counting Formulas}
\subsection{Single Split}
We can split $\Omega$ into two subgraphs at some non-vertex point $x_j=s_j$, $\Omega_1$ and $\Omega_2$, over which new conditions and operators are defined as in Remark \ref{rem:extraops}. In addition, we consider boundary conditions $\Gamma_1'$ and $\Gamma_2'$, which are identical to $\Gamma_1$ and $\Gamma_2$, respectively, at all vertices except $s_1$. In $\Gamma_1'$, we impose the Neumann condition $u'(s_1) = 0$, and in $\Gamma_2'$ we impose the Neumann condition $u_j'(s_1)=0$. Of course, from these new conditions we can construct operators $H^{\Omega_1}_{\Gamma_1'}$ and $H^{\Omega_2}_{\Gamma_2'}$, defined as follows:
\begin{equation}
\begin{split}
H^{\Omega_1}_{\Gamma_1'}u := H^{\Omega_1}u\text{ for }u\in \dom (H^{\Omega_1}_{\Gamma_1'}) := \{u\in H^2(\Omega_1): \gamma_{\Gamma_1'}u = \Vec{0}\},
\\H^{\Omega_2}_{\Gamma_2'}\vu := H^{\Omega_2}\vu\text{ for }\vu\in \dom (H^{\Omega_2}_{\Gamma_2'}) := \{\vu\in H^2(\Omega_2): \gamma_{\Gamma_2'}\vu = \Vec{0}\}.
\end{split}
\end{equation}

We will now demonstrate a way to count eigenvalues of $H^{\Omega}_{\Gamma}$ by instead working with $H^{\Omega_1}_{\Gamma_1}$ and $H^{\Omega_2}_{\Gamma_2}$.

\begin{remark}\label{rem:phitheta}
We will use the following functions satisfying Dirichlet and Neumann conditions throughout the rest of this section: $\Vec{\phi}$ and $\Vec{\theta}$. They are solutions to $(H^\Omega-\lambda)\vu=\Vec{0}$ and satisfy the following boundary conditions along the $j$th edge:
\begin{align}\label{eq:42}
\theta_j(s_1) = 1 \hspace{.5in}& \phi_j(s_1) = 0\nonumber \,,\\
\theta_j'(s_1) = 0\hspace{.5in} & \phi_j'(s_1) = 1\,.
\end{align}
For convenience, all other components (except for the jth components) for $\Vec{\phi}$ and $\Vec{\theta}$ are made identically $0$.
\end{remark}

\subsubsection{$M_1$ Construction}
When constructing $M_1$, we suppose that $\lambda\in \rho(H^{\Omega_1}_{\Gamma_1})$ and $f\in \CC$. Let $u\in H^2({\Omega_1})$ solve the inhomogeneous problem
\begin{equation}\label{eq:41}
\begin{cases}
(H^{\Omega_1}-\lambda)u = 0,\\
u(s_1) = f,\\
\text{$u$ satisfies the $\Gamma$ condition at $\ell_j$}.
\end{cases}
\end{equation}
Since $u$ does not satisfy the Dirichlet condition at $s_1$, we find that $u(x_j,\lambda)=cz_{j,j}|_{\Omega_1}(x_j,\lambda)$ for some $c$, where $z_{j,j}$ is the solution satisfying the $\beta$ conditions (and thus also $\beta^{\Gamma_1}$ conditions) on $\epsilon_j$ at $x_j=\ell_j$. Then, by equation \eqref{eq:300}, we have that
\begin{equation}\label{eq:44}
M_1(\lambda) = -\frac{u'(s_1,\lambda)}{u(s_1,\lambda)}= -\frac{cz_{j,j}|_{\Omega_1}'(s_1,\lambda)}{cz_{j,j}|_{\Omega_1}(s_1,\lambda)} = -\frac{z_{j,j}'(s_1,\lambda)}{z_{j,j}(s_1,\lambda)}\,.
\end{equation}
Of course, we can rewrite this as a quotient of Wronskians by utilizing $\phi_j$ and $\theta_j$ from equation \eqref{eq:42}:
\begin{equation}\label{eq:45}
M_1(\lambda) = -\frac{W_0(\theta_j(s_1,\lambda),z_{j,j}(s_1,\lambda))}{W_0(z_{j,j}(s_1,\lambda), \phi_j(s_1,\lambda))} = \frac{W_0(\theta_j(x_j,\lambda),z_{j,j}(x_j,\lambda))}{W_0( \phi_j(x_j,\lambda),z_{j,j}(x_j,\lambda))}\,.
\end{equation}
Of course, these Wronskians are simply the determinants of the fundamental matrices for the $\Gamma_1$ and $\Gamma_1'$ problems. Thus, rewritten as Evans functions, we have
\begin{equation}\label{eq:83}
M_1(\lambda) = \frac{E^{\Omega_1}_{\Gamma_1'}(\lambda)}{E^{\Omega_1}_{\Gamma_1}(\lambda)}.
\end{equation}

\subsubsection{$M_2$ Construction}\label{sec:lambda2}
The construction of $M_2$ as a quotient of Evans functions must use a different argument since it is defined by a problem over $\Omega_2$, which has a star structure.

\begin{theorem}\label{thm:map_evans}
Let $\lambda\in \rho(H^{\Omega_2}_{\Gamma_2})$. Then 
\[
M_2(\lambda) = -\frac{E^{\Omega_2}_{\Gamma_2'}(\lambda)}{E^{\Omega_2}_{\Gamma_2}(\lambda)}.
\]
\end{theorem}
\begin{proof}
Let $\lambda\in \rho(H^{\Omega_2}_{\Gamma_2})$. Let $\vu_{\Gamma_2,j}\in H^2({\Omega_2})$ be the unique solution to the problem
\begin{equation}\label{eq:46}
\begin{cases}
(H^{\Omega_2}-\lambda)\vu_{\Gamma_2,j} = \Vec{0},\\
\gamma_{\Gamma_2} \vu_{\Gamma_2,j} = \Vec{e}_{j}.
\end{cases}
\end{equation}
That is, $\vu_{\Gamma_2,j}$ satisfies all $\Gamma_2$ boundary conditions except the Dirichlet condition at $s_1$. Then by equation \eqref{eq:310}, we can express the map as follows:
\begin{equation}\label{eq:49}
M_2(\lambda) = \frac{(\vu_{\Gamma_2,j})'_j(s_1,\lambda)}{(\vu_{\Gamma_2,j})_j(s_1,\lambda)}=\frac{(\vu_{\Gamma_2,j})_j'(s_1,\lambda)}{1}\,.
\end{equation}
To find a more informative way of writing $(\vu_{\Gamma_2,j})'_j(s_1,\lambda)$, we can use Theorem \ref{thm:general_ugamma} with $i=j$, $\ell_j=s_1$, and fundamental $Y$ and $Z$ solutions defined on $\Omega_2$ using $\Gamma_2$ conditions.

We introduce $\delta^{\Gamma_2}_1$ and $\delta^{\Gamma_2}_2$, which are the $\delta$ matrices (see formula \eqref{eq:delta_def}) corresponding to $\Gamma_2$ boundary conditions. Due to the Dirichlet condition imposed at $s_1$, we have that the $j$th row and $j$th column of $\delta^{\Gamma_2}_1$ are both $\Vec{e}_j$, while the $j$th row and column of $\delta^{\Gamma_2}_2$ consist of zero entries. Therefore, $(\delta^{\Gamma_2}_1-i\delta^{\Gamma_2}_2)\Vec{e}_j = \Vec{e}_j$, which means that $(\delta^{\Gamma_2}_1-i\delta^{\Gamma_2}_2)^{-1}\Vec{e}_j = \Vec{e}_j$. Additionally, since $P_D$ is a projection onto $\ker \delta_2$, which $\Vec{e}_j$ is decidedly in, we have $P_D\Vec{e}_j = \Vec{e}_j$. However, since $P_D$, $\Tilde{P}_N$, and $\Tilde{P}_R$ are mutually orthogonal, we have that $\Tilde{P}_N\Vec{e}_j =\Tilde{P}_R\Vec{e}_j=\Vec{0}$. This means that $\Vec{e}_j$ is sent to zero by $L_j+M_j$, while it is passed through $N_j$ with nothing more than a sign swap. Hence, we can reduce Theorem \ref{thm:general_ugamma} for $i=j$ in this case to
\begin{equation}\label{eq:48}
(\vu_{\Gamma_2,j})_j(x_j,\lambda) = -\frac{z_{j,j}'(s_1,\lambda)(z_{j,j}(x_j,\lambda)|\mathcal{C}_j(\lambda,\Vec{e}_j)|-|C(\lambda)|y_{\tau_j,j}(x_j,\lambda))}{|C(\lambda)|D_j(\lambda)}\,.
\end{equation}
Note that in this formula, and for the rest of Section \ref{sec:lambda2}, $z_{j,j}$, $y_{\tau_j,j}$, $C$, $D_j$, etc. are all those objects related to the construction of $R_\lambda^{\Gamma_2}\vv$, as opposed to those from $R_\lambda^{\Gamma}\vv$. However, note that all of the functions except $z_{j,j}$ constructed using $\Gamma_2$ conditions are equal to their $\Gamma$ condition counterparts. 

The only portion of the Neumann trace we care about is $(\vu_{\Gamma_2,j})_j'(s_1,\lambda)$.
\begin{equation}\label{eq:62}
(\vu_{\Gamma_2,j})'_j(s_1,\lambda) = -\frac{z_{j,j}'(s_1,\lambda)(z_{j,j}'(s_1,\lambda)|\mathcal{C}_j(\lambda,\Vec{e}_j)|-|C(\lambda)|y_{\tau_j,j}'(s_1,\lambda))}{|C(\lambda)|D_j(\lambda)}\,.
\end{equation}
By Lemma \ref{thm:C_and_fundamental} we have that $|C| = (-1)^n|F|$, where $F$ is the fundamental matrix associated with the $H^{\Omega_2}_{\Gamma_2}$ operator. Note that $E^{\Omega_2}_{\Gamma_2}$ and $E^{\Omega_2}_{\Gamma_2'}$ are determinants of matrices identical to $F$ except for the $(n+j)$th column, which is replaced by $\Vec{\phi}$ and $\Vec{\theta}$, respectively. 

If we replace $z_{j,j}$ everywhere in $F$ and $C$ with $y_{\tau_j,j}$, we can again apply Lemma \ref{thm:C_and_fundamental} and get an analogous result. Of course, $C(\lambda)$ with all the $z_{j,j}$ contributions replaced by $y_{\tau_j,j}$ contributions is simply $\mathcal{C}_j(\lambda,\Vec{e}_j)$ (that is, $z_{j,j}$ and $z_{j,j}'$ are swapped out everywhere for $y_{\tau_j,j}$ and $y_{\tau_j,j}'$, respectively) as in equation \eqref{Ck}. Thus, we have that $|\mathcal{C}_j(\lambda,\Vec{e}_j)|=(-1)^n|\mathcal{F}_j(\lambda,\Vec{e}_j)|$ where $\mathcal{F}_j$ is $F$ with all $z_{j,j}$ contributions replaced with equivalent $y_{\tau_j,j}$ contributions. Then, equation \eqref{eq:62} can be rewritten as

\begin{align}\label{eq:80}
M_2(\lambda) &= (\vu_{\Gamma_2,j})'_j(s_1,\lambda)\nonumber 
\\&= -\frac{z_{j,j}'(s_1,\lambda)(z_{j,j}'(s_1,\lambda)(-1)^n|\mathcal{F}_j(\lambda,\Vec{e}_j)|-(-1)^n|F(\lambda)|y_{\tau_j,j}'(s_1,\lambda))}{(-1)^n|F(\lambda)|D_j(\lambda)}\nonumber\\
&=-\frac{z_{j,j}'(s_1,\lambda)(z_{j,j}'(s_1,\lambda)|\mathcal{F}_j(\lambda,\Vec{e}_j)|-|F(\lambda)|y_{\tau_j,j}'(s_1,\lambda))}{|F(\lambda)|D_j(\lambda)}\,.
\end{align}

Since there is a Dirichlet condition at $s_1$, we see that $z_{j,j}= \phi_j|_{\Omega_2}$ (the restriction of $\phi_j$ to $\Omega_2\cap \epsilon_j$), and thus $|F|=E^{\Omega_2}_{\Gamma_2}$. Observe that $|F|$ can be rewritten by expanding down the $\phi$ column when $x_j=0$ to give us:
\begin{equation}\label{eq:71}
|F(\lambda)| = \phi_j(0,\lambda) B_1(\lambda) - \phi_j'(0,\lambda)B_2(\lambda)\,,
\end{equation}
where $B_1$ and $B_2$ are the complementary minors picked up in Laplace expansion. Additionally, by our assumed initial conditions:
\begin{equation}\label{eq:70}
\theta_j(x_j) \phi_j'(x_j) - \theta_j'(x_j)\phi_j(x_j) = W_0(\theta_j,\phi_j) = W_0(\theta_j(s_1),\phi_j(s_1)) = \begin{vmatrix}
1 & 0 \\ 0 & 1
\end{vmatrix} = 1\,,
\end{equation}
and thus
\begin{equation}\label{eq:72}
\phi_j'(x_j) = \frac{1+\theta_j'(x_j)\phi_j(x_j)}{\theta_j(x_j)}\,.
\end{equation}
Then, by multiplying equation $\eqref{eq:80}$ by $-|F(\lambda)|$ we see:
\begin{align}\label{eq:75}
&-|F(\lambda)|M_2(\lambda)\nonumber=\frac{z_{j,j}'(s,\lambda)(z_{j,j}'(s_1,\lambda)|\mathcal{F}_j(\lambda,\Vec{e}_j)|-|F(\lambda)|y_{\tau_j,j}'(s_1,\lambda))}{D_j(\lambda)}\nonumber 
\\=& \frac{\phi_j'(s_1,\lambda)(\phi_j'(s_1,\lambda)|\mathcal{F}_j(\lambda,\Vec{e}_j)|-|F(\lambda)|y_{\tau_j,j}'(s_1,\lambda))}{W_0(y_{\tau_j,j}, z_{j,j})}\nonumber 
\\\myeq{\eqref{eq:42}}&\,\, \frac{|\mathcal{F}_j(\lambda,\Vec{e}_j)|-|F(\lambda)|y_{\tau_j,j}'(s_1,\lambda)}{W_0(y_{\tau_j,j}, \phi_j|_{\Omega_2})}\nonumber=\frac{|\mathcal{F}_j(\lambda,\Vec{e}_j)|-|F(\lambda)|\begin{vmatrix}
1& y_{\tau_j,j}(s_1,\lambda)\\
0 & y_{\tau_j,j}'(s_1,\lambda)
\end{vmatrix}}{y_{\tau_j,j}(0,\lambda)\phi_j'(0,\lambda)-y_{\tau_j,j}'(0,\lambda)\phi_j(0,\lambda)}\nonumber
\\ \myeq{\eqref{eq:71},\eqref{eq:72}}&\,\,\,\,\,\,\,\frac{|\mathcal{F}_j(\lambda,\Vec{e}_j)|-( \phi_j(0,\lambda) B_1(\lambda) - \phi_j'(0,\lambda)B_2(\lambda))\begin{vmatrix}
\theta_j(0,\lambda)& y_{\tau_j,j}(0,\lambda)\\
\theta_j'(0,\lambda) & y_{\tau_j,j}'(0,\lambda)
\end{vmatrix}}{y_{\tau_j,j}(0,\lambda)\frac{1+\theta_j'(0,\lambda)\phi_j(0,\lambda)}{\theta_j(0,\lambda)}-y_{\tau_j,j}'(0,\lambda)\phi_j(0,\lambda)}\nonumber
\\
&\text{(From here on out, all variables stay the same, so we omit dependencies)}\nonumber
\\=&\frac{|\mathcal{F}_j|-( \phi_j B_1 - \phi_j'B_2)
(\theta_j y_{\tau_j,j}'- y_{\tau_j,j}\theta_j')}{\frac{y_{\tau_j,j}}{\theta_j}+\phi_j(\frac{\theta_j'}{\theta_j}y_{\tau_j,j} - y_{\tau_j,j}')}\nonumber
\\=&\frac{|\mathcal{F}_j|-( \phi_j B_1 - \phi_j'B_2)
(-\theta_j)(y_{\tau_j,j}\frac{\theta_j'}{\theta_j}-y_{\tau_j,j}')}{\frac{y_{\tau_j,j}}{\theta_j}+\phi_j(\frac{\theta_j'}{\theta_j}y_{\tau_j,j} - y_{\tau_j,j}')}\nonumber
\\=&\frac{|\mathcal{F}_j|+  \theta_j B_1\phi_j(y_{\tau_j,j}\frac{\theta_j'}{\theta_j}-y_{\tau_j,j}') - \theta_j B_2\phi_j'(y_{\tau_j,j}\frac{\theta_j'}{\theta_j}-y_{\tau_j,j}')
}{\frac{y_{\tau_j,j}}{\theta_j}+\phi_j(\frac{\theta_j'}{\theta_j}y_{\tau_j,j} - y_{\tau_j,j}')}\nonumber
\\=&\frac{|\mathcal{F}_j|+  \theta_j B_1(\frac{y_{\tau_j,j}}{\theta_j}-\frac{y_{\tau_j,j}}{\theta_j}+\phi_j(y_{\tau_j,j}\frac{\theta_j'}{\theta_j}-y_{\tau_j,j}'))}{\frac{y_{\tau_j,j}}{\theta_j}+\phi_j(\frac{\theta_j'}{\theta_j}y_{\tau_j,j} - y_{\tau_j,j}')}\nonumber
\\&-\frac{B_2\theta_j(\frac{1+\theta_j'\phi_j}{\theta_j})(y_{\tau_j,j}\frac{\theta_j'}{\theta_j}-y_{\tau_j,j}')
}{\frac{y_{\tau_j,j}}{\theta_j}+\phi_j(\frac{\theta_j'}{\theta_j}y_{\tau_j,j} - y_{\tau_j,j}')}\nonumber
\\=&\theta_j B_1+\frac{|\mathcal{F}_j|  -\theta_j B_1\frac{y_{\tau_j,j}}{\theta_j} - B_2(y_{\tau_j,j}\frac{\theta_j'}{\theta_j}-y_{\tau_j,j}') - \theta_j'B_2\phi_j(y_{\tau_j,j}\frac{\theta_j'}{\theta_j}-y_{\tau_j,j}')
}{\frac{y_{\tau_j,j}}{\theta_j}+\phi_j(\frac{\theta_j'}{\theta_j}y_{\tau_j,j} - y_{\tau_j,j}')}\nonumber
\\=&\theta_j B_1 +\frac{|\mathcal{F}_j|  -\theta_j B_1\frac{y_{\tau_j,j}}{\theta_j} - B_2(y_{\tau_j,j}\frac{\theta_j'}{\theta_j}-y_{\tau_j,j}')
}{\frac{y_{\tau_j,j}}{\theta_j}+\phi_j(\frac{\theta_j'}{\theta_j}y_{\tau_j,j} - y_{\tau_j,j}')}\nonumber
\\&-\frac{\theta_j'B_2(\frac{y_{\tau_j,j}}{\theta_j}- \frac{y_{\tau_j,j}}{\theta_j}+\phi_j(y_{\tau_j,j}\frac{\theta_j'}{\theta_j}-y_{\tau_j,j}')
)}{\frac{y_{\tau_j,j}}{\theta_j}+\phi_j(\frac{\theta_j'}{\theta_j}y_{\tau_j,j} - y_{\tau_j,j}')}\nonumber
\\=&\theta_j B_1-\theta_j'B_2+\frac{|\mathcal{F}_j|  -\theta_j B_1\frac{y_{\tau_j,j}}{\theta_j} - B_2(y_{\tau_j,j}\frac{\theta_j'}{\theta_j}-y_{\tau_j,j}')+ \frac{y_{\tau_j,j}}{\theta_j}\theta_j'B_2
}{\frac{y_{\tau_j,j}}{\theta_j}+\phi_j(\frac{\theta_j'}{\theta_j}y_{\tau_j,j} - y_{\tau_j,j}')}\nonumber
\\=&\theta_j B_1-\theta_j'B_2+\frac{|\mathcal{F}_j|  - y_{\tau_j,j}B_1  +y_{\tau_j,j}'B_2}{\frac{y_{\tau_j,j}}{\theta_j}+\phi_j(\frac{\theta_j'}{\theta_j}y_{\tau_j,j} - y_{\tau_j,j}')}\,.
\end{align}
As discussed previously, $\mathcal{F}_j$ is simply $F$ with all the $z_{j,j}$ contributions swapped for $y_{\tau_j,j}$ contributions. By the definition of $B_1$ and $B_2$, the difference $y_{\tau_j,j}B_1- y_{\tau_j,j}'B_2$ is precisely the determinant of the fundamental solution matrix with $z_{j,j}$ swapped out for $y_{\tau_j,j}$. So, we certainly have $|\mathcal{F}_j|-(y_{\tau_j,j}B_1 - y_{\tau_j,j}'B_2)=0$, and thus we can represent $M_2$ as our desired quotient. Indeed, dividing equation \eqref{eq:75} by $-|F(\lambda)|$ tells us that
\begin{equation}\label{eq:79}
M_2(\lambda) = -\frac{\theta_j(0,\lambda)B_1(\lambda)-\theta_j'(0,\lambda)B_2(\lambda)}{|F(\lambda)|}=-\frac{E^{\Omega_2}_{\Gamma_2'}(\lambda)}{E^{\Omega_2}_{\Gamma_2}(\lambda)}\,.
\end{equation}
\end{proof}

\subsubsection{An Evans Function Relation}
We are now prepared to prove Theorem \ref{thm:determinant_equivalence}.
\begin{proof}
We know from equation \eqref{eq:83} and Theorem \ref{thm:map_evans} that we can represent $M_1$ and $M_2$ as the quotients of Evans functions:
\begin{align}\label{eq:76}
\begin{split}
M_1(\lambda) = \frac{E^{\Omega_1}_{\Gamma_1'}(\lambda)}{E^{\Omega_1}_{\Gamma_1}(\lambda)} \, ,
\\ M_2(\lambda) = -\frac{E^{\Omega_2}_{\Gamma_2'}(\lambda)}{E^{\Omega_2}_{\Gamma_2}(\lambda)} \, ,
\end{split}
\end{align}
Then, multiplying $M_1+M_2$ by $E^{\Omega_1}_{\Gamma_1}E^{\Omega_2}_{\Gamma_2}$ and applying equation \eqref{eq:76}, we see:
\begin{equation}\label{eq:77}
(M_1(\lambda)+M_2(\lambda))E^{\Omega_1}_{\Gamma_1}(\lambda)E^{\Omega_2}_{\Gamma_2}(\lambda) = E^{\Omega_1}_{\Gamma_1'}(\lambda)E^{\Omega_2}_{\Gamma_2}(\lambda) - E^{\Omega_2}_{\Gamma_2'}(\lambda)E^{\Omega_1}_{\Gamma_1}(\lambda) \,.
\end{equation}
Since $z_{j,j}(x_j,\lambda)$ is the initial value problem solution satisfying the $\Gamma$ condition associated with the vertex $x_j=\ell_j$, and $\Vec{\phi}$ and $\Vec{\theta}$ are as defined in Equation \eqref{eq:42}, we can evaluate the $2\times 2$ determinants $E^{\Omega_1}_{\Gamma_1}$ and $E^{\Omega_1}_{\Gamma_1'}$ at $x_j=s_1$ to obtain:
\begin{align}\label{eq:78}
&E^{\Omega_1}_{\Gamma_1'}(\lambda)E^{\Omega_2}_{\Gamma_2}(\lambda) - E^{\Omega_2}_{\Gamma_2'}(\lambda)E^{\Omega_1}_{\Gamma_1}(\lambda) \nonumber
\\=& \begin{vmatrix}
\theta_j(s_1,\lambda) & z_{j,j}(s_1,\lambda) \\
\theta_j(s_1,\lambda) & z_{j,j}'(s_1,\lambda)
\end{vmatrix}E^{\Omega_2}_{\Gamma_2}(\lambda) - E^{\Omega_2}_{\Gamma_2'}(\lambda)\begin{vmatrix}
\phi_j(s_1,\lambda) & z_{j,j}(s_1,\lambda) \\
\phi_j'(s_1,\lambda) & z_{j,j}'(s_1,\lambda)
\end{vmatrix}\nonumber
\\=& z_{j,j}'(s_1,\lambda)E^{\Omega_2}_{\Gamma_2}(\lambda)+E^{\Omega_2}_{\Gamma_2'}(\lambda)z_{j,j}(s_1,\lambda) \,.
\end{align}
We recall that $E^{\Omega_2}_{\Gamma_2}$ and $E^{\Omega_2}_{\Gamma_2'}$ are both determinants whose every column is equal to the the same column in the fundamental matrix $F$ for the $\Gamma$ problem except the $(n+j)$th, which is either $\Vec{\phi}$ or $\Vec{\theta}$, respectively. By evaluating these determinants at $x_j=s_1$, we can multiply the $z'_{j,j}$ and $z_{j,j}$ terms into the $(n+j)$th columns of $E_{\Gamma_2}$ and $E_{\Gamma_2'}$, respectively, to turn the expression from equation \eqref{eq:78} into the determinant of $F$. So
\begin{equation}\label{eq:84}
(M_1+M_2)E^{\Omega_1}_{\Gamma_1}E^{\Omega_2}_{\Gamma_2} = E^{\Omega_1}_{\Gamma_1'}E^{\Omega_2}_{\Gamma_2} - E^{\Omega_2}_{\Gamma_2'}E^{\Omega_1}_{\Gamma_1} = |F| = E^{\Omega}_{\Gamma}.
\end{equation}
Dividing this equation by $E^{\Omega_1}_{\Gamma_1}E^{\Omega_2}_{\Gamma_2}$  proves our desired equality.
\end{proof}
Since there is a one-to-one correspondence between the zeros of $E^{\Omega}_{\Gamma}$ and the eigenvalues of $H^{\Omega}_\Gamma$ (including algebraic multiplicities), Theorem \ref{thm:counting} clearly follows from Theorem \ref{thm:determinant_equivalence}.

\subsection{Double Split on One Wire}
\begin{figure}[H]
    \centering
    \includegraphics[width=1.5in]{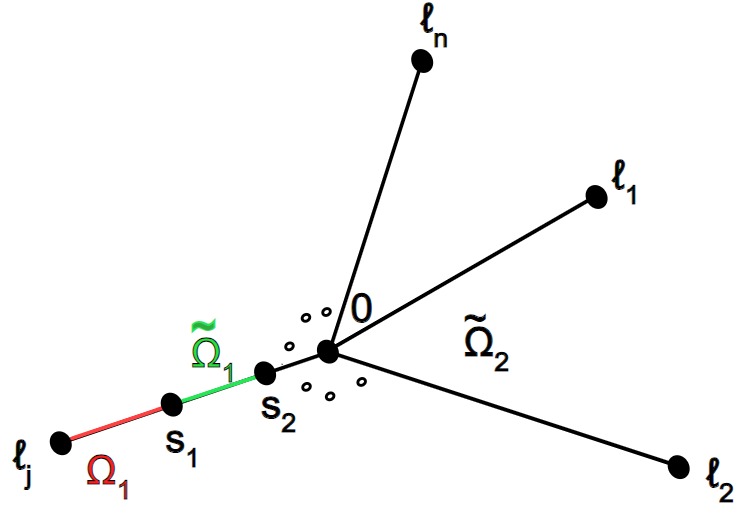}
    \caption{Splitting $\Omega$ into three subgraphs with two cuts on one wire.}
    \label{fig:10}
\end{figure}
So far our eigenvalue counting theorem is only proven for cases in which $\Omega_1$ and $\Omega_2$ meet at one point. We will now extend this result to cases where we split twice on one wire.

Let $s_1$ and $s_2$ be interior points of the $j$th wire of $\Omega$ such that $s_1>s_2$. We can split $\Omega$ into three subgraphs as in Definition \ref{def:secondsplit}. Then by Theorem $\ref{thm:determinant_equivalence}$, we have that
\begin{equation}\label{eq:111}
E^{\Omega}_\Gamma = E^{\Omega_1}_{\Gamma_1}E^{\Omega_2}_{\Gamma_2}(M_1+M_2),
\end{equation}
where $M_1+M_2$ is the two-sided map constructed at $s_1$. In a later proof, it will be beneficial to break up maps and Evans functions built for $\Omega_2$. This will require the introduction of some new boundary conditions.
\begin{definition}\label{def:moreops}
Let $\Tilde{\Gamma}_1^{DD}$, $\Tilde{\Gamma}_1^{DN}$, $\Tilde{\Gamma}_1^{ND}$, and $\Tilde{\Gamma}_1^{NN}$ be boundary conditions defined over $\Tilde{\Omega}_1$. A superscript of $D$ assigns the Dirichlet condition $u(s_i)=0$, a superscript of $N$ assigns the Neumann condition $u'(s_i) =0$, where $i$ corresponds to the position in which the letter sits. So for example, $\Tilde{\Gamma}_1^{DN}$ represents the boundary conditions $u(s_1)=u'(s_2)=0$. Note that $\Tilde{\Gamma}_1^{DD}$ is identical to the set of conditions called $\Tilde{\Gamma}_1$ in Definition \ref{def:secondsplit}. We also have $\Tilde{\Gamma}_2$ and $\Tilde{\Gamma}_2'$ defined over $\Tilde{\Omega}_2$, which are identical to $\Gamma_2$ at all common vertices. For $\Tilde{\Gamma}_2$, the Dirichlet condition $u_j(s_2) = 0$ is applied; for $\Tilde{\Gamma}_2'$, the Neumann condition $u_j'(s_2)=0$ is applied.
\end{definition}

\begin{remark}\label{rem:opdef}
At this point we have a standard way of building operators and Evans functions based on sets of boundary conditions on subgraphs. Given boundary conditions $\Gamma_0$ defined over some subgraph $\Omega^*$, we can let $H^{\Omega^*}$ be the restriction of $H^{\Omega}$ to $\Omega^*$. Then we can define the operator $H^{\Omega^*}_{\Gamma_0}$ as follows:
\[
H^{\Omega^*}_{\Gamma_0}\vu := H^{\Omega^*}\vu\text{ for }\vu\in \dom (H^{\Omega^*}_{\Gamma_0}) := \{\vu\in H^2(\Omega^*): \gamma_{\Gamma_0}\vu = \Vec{0}\}.
\]
We let the Evans function associated with $H^{\Omega^*}_{\Gamma_0}$ be denoted by $E^{\Omega^*}_{\Gamma_0}$.
\end{remark}

With this new notation, we can split $\Omega_2$ at $s_2$ to get the three desired subgraphs, and another application of Theorem $\ref{thm:determinant_equivalence}$ gives us that
\begin{equation}\label{eq:112}
E^{\Omega_2}_{\Gamma_2} = E^{\Tilde{\Omega}_1}_{\Tilde{\Gamma}_1^{DD}}E^{\Tilde{\Omega}_2}_{\Tilde{\Gamma}_2}(\Tilde{M}_1+\Tilde{M}_2),
\end{equation}
where $\Tilde{M}_1+\Tilde{M}_2$ is a two sided map defined at $s_2$ with a Dirichlet condition at $s_1$.

Combining equations \eqref{eq:111} and \eqref{eq:112}, we find that:
\begin{align}\label{eq:250}
\begin{split}
E^{\Omega}_\Gamma &= E^{\Omega_1}_{\Gamma_1}E^{\Omega_2}_{\Gamma_2}(M_1+M_2)\\
&=E^{\Omega_1}_{\Gamma_1}E^{\Tilde{\Omega}_1}_{\Tilde{\Gamma}_1^{DD}}E^{\Tilde{\Omega}_2}_{\Tilde{\Gamma}_2}(\Tilde{M}_1+\Tilde{M}_2)(M_1+M_2).
\end{split}
\end{align}
We will now define the maps $\M_1$ and $\M_2$, the two-dimensional extensions of the one-sided maps. 
\begin{definition}\label{def:onewireMs}
Let $\lambda\in \rho(H^{\Tilde{\Omega}_1}_{\Tilde{\Gamma}_1^{DD}})$ and fix $\vf\in \CC^2$. Then, we define $\M_1(\lambda):\CC^2\rightarrow \CC^2$ as follows:
\[
\M_1(\lambda)\vf := \begin{bmatrix}
u'(s_1,\lambda)\\
-u'(s_2,\lambda)
\end{bmatrix},
\]
where $u$ is the unique solution to the boundary value problem
\[
\begin{cases}
(H^{\Tilde{\Omega}_1}-\lambda)u = 0,\\
\gamma_{\Tilde{\Gamma}_1^{DD}}u = \vf.
\end{cases}
\]
Likewise, for $\lambda\in \rho(H^{\Omega_1}_{\Gamma_1})\cap\rho(H^{\Tilde{\Omega}_2}_{\Tilde{\Gamma}_2})$ and $\vf\in \CC^2$, we can instead define $\M_2(\lambda):\CC^2\rightarrow \CC^2$ as follows
\[
\M_2(\lambda)\vf :=
\begin{bmatrix}
-u'(s_1,\lambda)\\
u_j'(s_2,\lambda)
\end{bmatrix},
\]
where $u$ is the unique solution to the boundary value problem
\[
\begin{cases}
(H^{\Omega_1}-\lambda)u = 0,\\
\gamma_{\Gamma_1}u = \begin{bmatrix}
0\\f_1    
\end{bmatrix}.
\end{cases}
\]
and $u_j$ is the $j$th component of the unique solution $\vu$ to the boundary value problem
\[
\begin{cases}
(H^{\Tilde{\Omega}_2}-\lambda)\vu = \Vec{0},\\
u_j(s_2) = f_2,\\
\text{$\vu$ satisfies all other $\Tilde{\Gamma}_2$ conditions.}
\end{cases}
\]
\end{definition}

\begin{remark}
$\M_1$ and $\M_2$ are meant to be $2\times 2$ analogs of the one-sided maps $M_1$ and $M_2$. We can utilize some of the previously constructed maps to rewrite construct $\M_2$ as follows:
\begin{equation}\label{eq:120}
\M_2(\lambda) := \begin{bmatrix}
M_1(\lambda) & 0\\
0 & \Tilde{M}_2(\lambda)
\end{bmatrix}.
\end{equation}
Additionally, we can get a more concrete handle on $\M_1$ using the following construction
\begin{equation}\label{eq:124}
\M_1(\lambda) := \begin{bmatrix}
\frac{u'(s_1,\lambda)}{u(s_1,\lambda)} & \frac{w'(s_1,\lambda)}{w(s_2,\lambda)}\\
-\frac{u'(s_2,\lambda)}{u(s_1,\lambda)} & -\frac{w'(s_2,\lambda)}{w(s_2,\lambda)}
\end{bmatrix},
\end{equation}
where $u$ and $w$ are the solutions to the boundary value problems:
\begin{equation}\label{eq:122}
\begin{split}
\begin{cases}
(H^{\Tilde{\Omega}_1}-\lambda)u = 0,\\
u(s_1) = 1,\\
u(s_2) = 0,
\end{cases}
\\
\begin{cases}
(H^{\Tilde{\Omega}_1}-\lambda)w = 0,\\
w(s_1) = 0,\\
w(s_2) = 1.
\end{cases}
\end{split}
\end{equation}
Note that these constructions of $\M_1$ and $\M_2$ are identical to those produced by the the definitions.
\end{remark}

\begin{theorem}\label{thm:twosplitonewire}
Suppose $\Omega$ is split (as described in Definition \ref{def:secondsplit}) into three quantum subgraphs $\Omega_1$, $\Tilde{\Omega}_1$, and $\Tilde{\Omega}_2$, at the points $0<s_2<s_1<\ell_j$ for some $1\leq j\leq n$. Additionally, suppose that $\lambda\in\rho(H^{\Omega_1}_{\Gamma_1})\cap \rho(H^{\Tilde{\Omega}_1}_{\Tilde{\Gamma}_1})\cap \rho(H^{\Tilde{\Omega}_2}_{\Tilde{\Gamma}_2})$, and let $\M_1+\M_2$ be the two-sided map associated with $s_1$ and $s_2$. Then
\begin{equation}
\frac{E^{\Omega}_{\Gamma}(\lambda)}{E^{\Omega_1}_{\Gamma_1}(\lambda)E^{\Tilde{\Omega}_1}_{\Tilde{\Gamma}_1}(\lambda)E^{\Tilde{\Omega}_2}_{\Tilde{\Gamma}_2}(\lambda)} = |(\M_1+\M_2)(\lambda)|,
\end{equation}
where $E^{\Omega}_{\Gamma}$, $E^{\Omega_1}_{\Gamma_1}$, $E^{\Tilde{\Omega}_1}_{\Tilde{\Gamma}_1}$, and $E^{\Tilde{\Omega}_2}_{\Tilde{\Gamma}_2}$ are the Evans functions associated with the operators $H^{\Omega}_{\Gamma}$, $H^{\Omega_1}_{\Gamma_1}$, $H^{\Tilde{\Omega}_1}_{\Tilde{\Gamma}_1}$, and $H^{\Tilde{\Omega}_2}_{\Tilde{\Gamma}_2}$, respectively.
\end{theorem}
\begin{proof}
The desired result can be obtained from equation \eqref{eq:250}. We claim that the corrective term $(\Tilde{M}_1+\Tilde{M}_2)(M_1+M_2)$ is precisely the determinant of the two-sided map $\M_1 + \M_2$. First, we construct several key functions. Let $\phi_1,\phi_2,\theta_1,\theta_2$ all be solutions to $(H^{\Tilde{\Omega}_1}-\lambda)u=0$ defined over $\Tilde{\Omega}_1$, which satisfy the following conditions:
\begin{align}\label{eq:121}
& \phi_1(s_1) = 0 & \theta_1(s_1) = 1 & &\phi_2(s_2) = 0 && \theta_2(s_2) = 1,\nonumber\\
& \phi_1'(s_1) = 1 & \theta_1'(s_1) = 0 && \phi_2'(s_2) = 1 && \theta_2'(s_2) = 0.
\end{align}

Consider the functions $u$ and $w$ from equation \eqref{eq:122}. Since $u$ satisfies the Dirichlet condition at $s_2$ but not $s_1$, it is parallel to $\phi_2$, so there exists $\sigma_2$ such that $u(x_j) = \sigma_2\phi_2(x_j)$. Similarly, since $w$ satisfies Dirichlet conditions at $s_1$ by not $s_2$, there exists some $\sigma_1$ such that $w(x_j) = \sigma_1\phi_1(x_j)$. Then we can rewrite $\M_1$ as
\begin{align}\label{eq:125}
\M_1(\lambda) &= \begin{bmatrix}
\frac{u'(s_1,\lambda)}{u(s_1,\lambda)} & \frac{w'(s_1,\lambda)}{w(s_2,\lambda)}\\
-\frac{u'(s_2,\lambda)}{u(s_1,\lambda)} & -\frac{w'(s_2,\lambda)}{w(s_2,\lambda)}
\end{bmatrix} = \begin{bmatrix}
\frac{\phi_2'(s_1,\lambda)}{\phi_2(s_1,\lambda)} & \frac{\phi_1'(s_1,\lambda)}{\phi_1(s_2,\lambda)}\\
-\frac{\phi_2'(s_2,\lambda)}{\phi_2(s_1,\lambda)} & -\frac{\phi_1'(s_2,\lambda)}{\phi_1(s_2,\lambda)}
\end{bmatrix}\nonumber\\
&= \begin{bmatrix}
-\frac{W(\theta_1,\phi_2)}{W(\phi_1,\phi_2)} & \frac{1}{W(\phi_1,\phi_2)}\\
\frac{1}{W(\phi_1,\phi_2)} & \frac{W(\phi_1,\theta_2)}{W(\phi_1,\phi_2)}.
\end{bmatrix} = \begin{bmatrix}
-\frac{W(\phi_2,\theta_1)}{W(\phi_2,\phi_1)} & -\frac{1}{W(\phi_2,\phi_1)}\\
-\frac{1}{W(\phi_2,\phi_1)} & \frac{W(\theta_2,\phi_1)}{W(\phi_2,\phi_1)}
\end{bmatrix}.
\end{align}
We quickly note that, by definition,
\begin{align*}
&E^{\Tilde{\Omega}_1}_{\Tilde{\Gamma}_1^{DD}}  = W(\phi_2,\phi_1)
&E^{\Tilde{\Omega}_1}_{\Tilde{\Gamma}_1^{DN}} = W(\theta_2, \phi_1),\\
&E^{\Tilde{\Omega}_1}_{\Tilde{\Gamma}_1^{ND}}  = W(\phi_2,\theta_1)
&E^{\Tilde{\Omega}_1}_{\Tilde{\Gamma}_1^{NN}}  = W(\theta_2,\theta_1).
\end{align*}
By equation \eqref{eq:83} adjusted to a subgraph $\Tilde{\Omega}_1$, we also have that
\begin{equation}\label{eq:127}
\Tilde{M}_1 = \frac{E^{\Tilde{\Omega}_1}_{\Tilde{\Gamma}_1^{DN}}}{E^{\Tilde{\Omega}_1}_{\Tilde{\Gamma}_1^{DD}}} = \frac{W(\theta_2,\phi_1)}{W(\phi_2,\phi_1)}.
\end{equation}
We additionally note that $\Tilde{M}_1=-\frac{w'(s_2,\lambda)}{w(s_2,\lambda)}$, and can be substituted as such into $\M_1$. The final note we make is that 
\begin{equation}\label{eq:129}
1=\begin{vmatrix}
1 & 0 \\
0 & 1
\end{vmatrix} = \begin{vmatrix}
\theta_2(s_2) & \phi_2(s_2)\\
\theta_2'(s_2) & \phi_2'(s_2)
\end{vmatrix} = W(\theta_2,\phi_2) = \theta_2\phi_2'-\theta_2'\phi_2.
\end{equation}
\\
We can now begin calculating the determinant of the two-sided map.
\begin{align}\label{eq:126}
&|\M_1+\M_2| \,\,\,\,\,\myeq{\eqref{eq:120},\eqref{eq:125}}\,\,\,\,\, \begin{vmatrix}
-\frac{W(\phi_2,\theta_1)}{W(\phi_2,\phi_1)}+M_1 & -\frac{1}{W(\phi_2,\phi_1)}\\
-\frac{1}{W(\phi_2,\phi_1)} & \Tilde{M}_1+\Tilde{M}_2
\end{vmatrix}\nonumber\\
&=(-\frac{W(\phi_2,\theta_1)}{W(\phi_2,\phi_1)}+M_1)(\Tilde{M}_1+\Tilde{M}_2)- \frac{1}{W(\phi_2,\phi_1)^2}\nonumber\\
&=-\frac{W(\phi_2,\theta_1)}{W(\phi_2,\phi_1)}(\Tilde{M}_1+\Tilde{M}_2)+M_1(\Tilde{M}_1+\Tilde{M}_2)- \frac{1}{W(\phi_2,\phi_1)^2}\nonumber\\
&\myeq{\eqref{eq:127}}-\frac{W(\phi_2,\theta_1)}{W(\phi_2,\phi_1)}\frac{W(\theta_2,\phi_1)}{W(\phi_2,\phi_1)}-\frac{W(\phi_2,\theta_1)}{W(\phi_2,\phi_1)}\Tilde{M}_2+M_1(\Tilde{M}_1+\Tilde{M}_2)- \frac{1}{W(\phi_2,\phi_1)^2}\nonumber\\
&=-\frac{E^{\Tilde{\Omega}_1}_{\Tilde{\Gamma}_1^{ND}}}{E^{\Tilde{\Omega}_1}_{\Tilde{\Gamma}_1^{DD}}}\Tilde{M}_2+M_1(\Tilde{M}_1+\Tilde{M}_2)-\frac{1}{E^{\Tilde{\Omega}_1}_{\Tilde{\Gamma}_1^{DD}}}\left (\frac{W(\phi_2,\theta_1)W(\theta_2,\phi_1)+1}{W(\phi_2,\phi_1)} \right )\nonumber\\
&=-\frac{E^{\Tilde{\Omega}_1}_{\Tilde{\Gamma}_1^{ND}}}{E^{\Tilde{\Omega}_1}_{\Tilde{\Gamma}_1^{DD}}}\Tilde{M}_2+M_1(\Tilde{M}_1+\Tilde{M}_2)-\frac{1}{E^{\Tilde{\Omega}_1}_{\Tilde{\Gamma}_1^{DD}}}\left (\frac{(-\phi_2'(s_1))\theta_2(s_1)+1}{\phi_2(s_1)} \right )\nonumber\\
&\myeq{\eqref{eq:129}}-\frac{E^{\Tilde{\Omega}_1}_{\Tilde{\Gamma}_1^{ND}}}{E^{\Tilde{\Omega}_1}_{\Tilde{\Gamma}_1^{DD}}}\Tilde{M}_2+M_1(\Tilde{M}_1+\Tilde{M}_2)-\frac{1}{E^{\Tilde{\Omega}_1}_{\Tilde{\Gamma}_1^{DD}}}\left (\frac{-\phi_2(s_1)\theta_2'(s_1)}{\phi_2(s_1)} \right )\nonumber\\
&=-\frac{E^{\Tilde{\Omega}_1}_{\Tilde{\Gamma}_1^{ND}}}{E^{\Tilde{\Omega}_1}_{\Tilde{\Gamma}_1^{DD}}}\Tilde{M}_2+M_1(\Tilde{M}_1+\Tilde{M}_2)-\frac{E^{\Tilde{\Omega}_1}_{\Tilde{\Gamma}_1^{NN}}}{E^{\Tilde{\Omega}_1}_{\Tilde{\Gamma}_1^{DD}}}.
\end{align}
On the other hand, we can also simplify the corrective term. By Theorem \ref{thm:map_evans}, we have that $M_2 = -\frac{E^{\Omega_2}_{\Gamma_2'}}{E^{\Omega_2}_{\Gamma_2}}$, which can be further decomposed. In particular, another application of Theorem \ref{thm:determinant_equivalence} allows us to obtain
\begin{equation}\label{eq:130}
E^{\Omega_2}_{\Gamma_2'} = E^{\Tilde{\Omega}_1}_{\Tilde{\Gamma}_1^{ND}}E^{\Tilde{\Omega}_2}_{\Tilde{\Gamma}_2}(\hat{M}_1+\hat{M}_2),
\end{equation}
where $\hat{M}_1+\hat{M}_2$ is the associated two-sided map at $s_2$ (here there is a Neumann condition at $s_1$). Observe that $\hat{M}_2=\Tilde{M}_2$ (the reasoning is clear from Figure \ref{fig:10}). Thus, by equations \eqref{eq:112} and \eqref{eq:130} we have
\begin{equation}\label{eq:131}
M_2 = -\frac{E^{\Omega_2}_{\Gamma_2'}}{E^{\Omega_2}_{\Gamma_2}}= -\frac{E^{\Tilde{\Omega}_1}_{\Tilde{\Gamma}_1^{ND}}E^{\Tilde{\Omega}_2}_{\Tilde{\Gamma}_2}(\hat{M}_1+\hat{M}_2)}{E^{\Tilde{\Omega}_1}_{\Tilde{\Gamma}_1^{DD}}E^{\Tilde{\Omega}_2}_{\Tilde{\Gamma}_2}(\Tilde{M}_1+\Tilde{M}_2)} =-\frac{E^{\Tilde{\Omega}_1}_{\Tilde{\Gamma}_1^{ND}}(\hat{M}_1+\Tilde{M}_2)}{E^{\Tilde{\Omega}_1}_{\Tilde{\Gamma}_1^{DD}}(\Tilde{M}_1+\Tilde{M}_2)} .
\end{equation}
Finally, noting that, by equation \eqref{eq:83} adjusted to a subgraph $\Tilde{\Omega}_1$, $\Hat{M}_1$ can be rewritten as Evans functions as follows:
\begin{equation}\label{eq:128}
\Hat{M}_1 = \frac{E^{\Tilde{\Omega}_1}_{\Tilde{\Gamma}_1^{NN}}}{E^{\Tilde{\Omega}_1}_{\Tilde{\Gamma}_1^{ND}}}= \frac{W(\theta_2,\theta_1)}{W(\phi_2,\theta_1)}.
\end{equation}
We can decompose the corrective term into a workable state. Indeed,
\begin{align}\label{eq:132}
(M_1+M_2)(\Tilde{M}_1+\Tilde{M}_2) &\myeq{\eqref{eq:131}} \left ( M_1-\frac{E^{\Tilde{\Omega}_1}_{\Tilde{\Gamma}_1^{ND}}(\hat{M}_1+\Tilde{M}_2)}{E^{\Tilde{\Omega}_1}_{\Tilde{\Gamma}_1^{DD}}(\Tilde{M}_1+\Tilde{M}_2)}\right )(\Tilde{M}_1+\Tilde{M}_2 )\nonumber\\
&=M_1(\Tilde{M}_1+\Tilde{M}_2)-\frac{E^{\Tilde{\Omega}_1}_{\Tilde{\Gamma}_1^{ND}}(\hat{M}_1+\Tilde{M}_2)}{E^{\Tilde{\Omega}_1}_{\Tilde{\Gamma}_1^{DD}}}\nonumber\\
&=M_1(\Tilde{M}_1+\Tilde{M}_2) -\frac{E^{\Tilde{\Omega}_1}_{\Tilde{\Gamma}_1^{ND}}\Tilde{M}_2}{E^{\Tilde{\Omega}_1}_{\Tilde{\Gamma}_1^{DD}}} -\frac{E^{\Tilde{\Omega}_1}_{\Tilde{\Gamma}_1^{ND}}\hat{M}_1}{E^{\Tilde{\Omega}_1}_{\Tilde{\Gamma}_1^{DD}}}\nonumber\\
&\myeq{\eqref{eq:128}}M_1(\Tilde{M}_1+\Tilde{M}_2) -\frac{E^{\Tilde{\Omega}_1}_{\Tilde{\Gamma}_1^{ND}}\Tilde{M}_2}{E^{\Tilde{\Omega}_1}_{\Tilde{\Gamma}_1^{DD}}} -\frac{E^{\Tilde{\Omega}_1}_{\Tilde{\Gamma}_1^{ND}}\frac{E^{\Tilde{\Omega}_1}_{\Tilde{\Gamma}_1^{NN}}}{E^{\Tilde{\Omega}_1}_{\Tilde{\Gamma}_1^{ND}}}}{E^{\Tilde{\Omega}_1}_{\Tilde{\Gamma}_1^{DD}}}\nonumber\\
&=M_1(\Tilde{M}_1+\Tilde{M}_2) -\frac{E^{\Tilde{\Omega}_1}_{\Tilde{\Gamma}_1^{ND}}\Tilde{M}_2}{E^{\Tilde{\Omega}_1}_{\Tilde{\Gamma}_1^{DD}}} -\frac{E^{\Tilde{\Omega}_1}_{\Tilde{\Gamma}_1^{NN}}}{E^{\Tilde{\Omega}_1}_{\Tilde{\Gamma}_1^{DD}}}.
\end{align}
Comparing equations \eqref{eq:126} and \eqref{eq:132}, we indeed find that $|\M_1+\M_2| = (M_1+M_2)(\Tilde{M}_1+\Tilde{M}_2)$ as claimed. This means we can substitute $|\M_1+\M_2|$ into equation \eqref{eq:250}, completing the proof of this theorem.
\end{proof}

\subsection{Splitting on Two Wires}
\begin{figure}[H]
    \centering
    \includegraphics[width=1.5in]{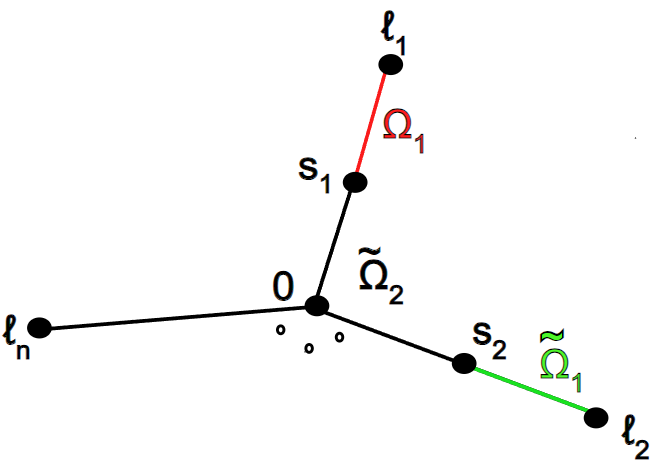}
    \caption{Splitting $\Omega$ into three subgraphs with two cuts on seperate wires.}
    \label{fig:11}
\end{figure}
Suppose we wish to split $\Omega$ into three subgraphs, where the cuts are made on two different wires. We can extend our Evans function equivalence to this case as well.

Without loss of generality, we assume that these cuts are made on the first and second wires at vertices $s_1$ and $s_2$ satisfying $0< s_1< \ell_1$ and $0< s_2< \ell_2$. The first split is done on the first wire, splitting $\Omega$ into $\Omega_1$ and $\Omega_2$ . We define our usual $\Gamma_1$, $\Gamma_2$, $\Gamma_1'$, and $\Gamma_2'$ conditions as described in Remark \ref{rem:extraops}. The second split is done on the second wire, splitting $\Omega_2$ into $\Tilde{\Omega}_1$ and $\Tilde{\Omega}_2$. Over these subgraphs we define six sets of boundary conditions: $\Tilde{\Gamma}_1$, $\Tilde{\Gamma}_1'$, $\Tilde{\Gamma}_2^{DD}$, $\Tilde{\Gamma}_2^{DN}$, $\Tilde{\Gamma}_2^{ND}$, and $\Tilde{\Gamma}_2^{NN}$. $\Tilde{\Gamma}_1$ and $\Tilde{\Gamma}_1'$ both include the $\Gamma_2$ condition at $\ell_2$, but $\Tilde{\Gamma}_1$ includes the Dirichlet condition $u_2(s_2)=0$ at $s_2$, while $\Tilde{\Gamma}_1'$ includes the Neumann condition $u_2'(s_2)=0$. The various $\Tilde{\Gamma}_2$ sets share $\Gamma_2$ conditions at all common vertices between $\Omega_2$ and $\Tilde{\Omega}_2$ except $s_1$. The conditions at $s_i$ are either Dirichlet $u_i(s_i)=0$ (associated with superscript $D$) or Neumann $u_i'(s_i)$ (associated with superscript $N$), where $i$ is determined by the position of the signifying letter in the superscript. For example, $\Tilde{\Gamma}_2^{ND}$ includes a Neumann condition at $s_1$ and a Dirichlet condition at $s_2$. Note that $\Tilde{\Gamma}_2^{DD}$ is exactly the same as $\Tilde{\Gamma}_2$ as defined in Definition \ref{def:secondsplit}. We define $H$ operators and corresponding Evans functions for all of these boundary conditions according to Remark \ref{rem:opdef}.

We observe that two applications of Theorem \ref{thm:determinant_equivalence} give us
\begin{equation}\label{eq:502}
E^{\Omega}_\Gamma = E^{\Omega_1}_{\Gamma_1}E^{\Omega_2}_{\Gamma_2}(M_1+M_2),
\end{equation}
and
\begin{equation}\label{eq:501}
E^{\Omega_2}_{\Gamma_2} = E^{\Tilde{\Omega}_1}_{\Tilde{\Gamma}_1}E^{\Tilde{\Omega}_2}_{\Tilde{\Gamma}_2^{DD}}(\Tilde{M}_1+\Tilde{M}_2),
\end{equation}
which combine to give us
\[\label{eq:170}
E^{\Omega}_\Gamma= E^{\Omega_1}_{\Gamma_1}E^{\Tilde{\Omega}_1}_{\Tilde{\Gamma}_1}E^{\Tilde{\Omega}_2}_{\Tilde{\Gamma}_2^{DD}}(\Tilde{M}_1+\Tilde{M}_2)(M_1+M_2).
\]

As in the previous section, we can extend our counting equivalence if the following equality holds: $|\M_1+\M_2| = (\Tilde{M}_1+\Tilde{M}_2)(M_1+M_2)$. Of course, for this new case we need to slightly adjust the definitions of $\M_1$ and $\M_2$ from those in Definition \ref{def:onewireMs}.

\begin{definition}\label{def:twowireMs}
Let $\lambda\in \rho(H^{\Omega_1}_{\Gamma_1})\cap\rho(H^{\Tilde{\Omega}_1}_{\Tilde{\Gamma}_1})$ and $\vf\in \CC^2$. Then we can define $\M_1(\lambda):\CC^2\rightarrow \CC^2$ as follows:
\[
\M_1(\lambda)\vf :=
\begin{bmatrix}
-u_1'(s_1,\lambda)\\
-u_2'(s_2,\lambda)
\end{bmatrix},
\]
where $\vu$ solves 
\[
\begin{cases}
(H^{\Omega_1}-\lambda)u_1 = 0,\\
(H^{\Tilde{\Omega}_1}-\lambda)u_2 = 0,\\
u_1(s_1)=f_1,\\
u_2(s_2)=f_2,\\
\text{$\vu$ satisfies all other $\Gamma_1$ and $\Tilde{\Gamma}_1$ conditions}.
\end{cases}
\]
Note that in this case $\vu$ consists of only two components, since its domain is two disconnected wires.

Likewise, for $\lambda\in \rho(H^{\Tilde{\Omega}_2}_{\Tilde{\Gamma}_2^{DD}})$ and $\vf\in \CC^2$, we define $\M_2(\lambda):\CC^2\rightarrow \CC^2$ as follows:
\[
\M_2(\lambda)\vf := \begin{bmatrix}
u_1'(s_1,\lambda)\\
u_2'(s_2,\lambda)
\end{bmatrix},
\]
where $\vu$ is the unique solution to the boundary value problem
\[
\begin{cases}
(H^{\Tilde{\Omega}_2}-\lambda)\vu = \Vec{0},\\
u_1(s_1)=f_1,\\
u_2(s_2)=f_2,\\
\text{$\vu$ satisfies all other $\Gamma_2^{DD}$ conditions}.
\end{cases}
\]
\end{definition}

\begin{remark}\label{rem:twowireMs}
We can find a helpful matrix representation of $\M_2$ by observing its action on the standard $\CC^2$ basis vectors $\Vec{e}_1$ and $\Vec{e}_2$. That is,
to construct $\M_2$, we introduce the functions $\vu$ and $\vw$ over $\Tilde{\Omega}_2$ that represent the unique solutions to the following boundary value problems:
\begin{equation}\label{eq:171}
\begin{cases}
(H^{\Tilde{\Omega}_2}-\lambda)\vu = \Vec{0},\\
u_1(s_1) = 1,\\
u_2(s_2) = 0,\\
\text{$\vu$ satisfies all $\Gamma$ conditions at the origin},
\end{cases}
\end{equation}
and
\begin{equation}\label{eq:172}
\begin{cases}
(H^{\Tilde{\Omega}_2}-\lambda)\vw = \Vec{0},\\
w_1(s_1) = 0,\\
w_2(s_2) = 1,\\
\text{$\vw$ satisfies all $\Gamma$ conditions at the origin}.
\end{cases}
\end{equation}
Then $\M_2$ can be constructed as follows:
\begin{equation}\label{eq:173}
\M_2(\lambda) = \begin{bmatrix}
\frac{u_1'(s_1,\lambda)}{u_1(s_1,\lambda)} & \frac{w_1'(s_1,\lambda)}{w_2(s_2,\lambda)}\\
\frac{u_2'(s_2,\lambda)}{u_1(s_1,\lambda)} & \frac{w_2'(s_2,\lambda)}{w_2(s_2,\lambda)}
\end{bmatrix}.
\end{equation}
Also, we will build $\M_1$ in terms of maps already encountered:
\[\label{eq:504}
\M_1(\lambda) = \begin{bmatrix}
M_1(\lambda) & 0\\
0 & \Tilde{M}_1(\lambda)
\end{bmatrix}.
\]
\end{remark}
We now have all the necessary objects defined, and can prove our extension of Theorem \ref{thm:determinant_equivalence}.

\begin{theorem}\label{thm:twosplittwowire}
Suppose $\Omega$ is split (as described in Definition \ref{def:secondsplit}) into three quantum subgraphs $\Omega_1$, $\Tilde{\Omega}_1$, and $\Tilde{\Omega}_2$, at non-vertex points $s_1$ and $s_2$ on seperate edges of $\Omega$. Additionally, suppose that $\lambda\in\rho(H^{\Omega_1}_{\Gamma_1})\cap \rho(H^{\Tilde{\Omega}_1}_{\Tilde{\Gamma}_1})\cap \rho(H^{\Tilde{\Omega}_2}_{\Tilde{\Gamma}_2})$, and let $\M_1+\M_2$ be the two-sided map associated with $s_1$ and $s_2$. Then
\begin{equation}
\frac{E^{\Omega}_{\Gamma}(\lambda)}{E^{\Omega_1}_{\Gamma_1}(\lambda)E^{\Tilde{\Omega}_1}_{\Tilde{\Gamma}_1}(\lambda)E^{\Tilde{\Omega}_2}_{\Tilde{\Gamma}_2}(\lambda)} = |(\M_1+\M_2)(\lambda)|,
\end{equation}
where $E^{\Omega}_{\Gamma}$, $E^{\Omega_1}_{\Gamma_1}$, $E^{\Tilde{\Omega}_1}_{\Tilde{\Gamma}_1}$, and $E^{\Tilde{\Omega}_2}_{\Tilde{\Gamma}_2}$ are the Evans functions associated with the operators $H^{\Omega}_{\Gamma}$, $H^{\Omega_1}_{\Gamma_1}$, $H^{\Tilde{\Omega}_1}_{\Tilde{\Gamma}_1}$, and $H^{\Tilde{\Omega}_2}_{\Tilde{\Gamma}_2}$, respectively.
\end{theorem}

\begin{proof}
Consider the $\vw$ and $\vu$ functions as discussed in Remark \ref{rem:twowireMs}. We note that, by Theorem \ref{thm:map_evans} adjusted to a subgraph $\Tilde{\Omega}_2$, 
\begin{equation}\label{eq:516}
\frac{w_2'(s_2,\lambda)}{w_2(s_2,\lambda)}=\Tilde{M}_2(\lambda)=-\frac{E^{\Tilde{\Omega}_2}_{\Tilde{\Gamma}_2^{DN}}(\lambda)}{E^{\Tilde{\Omega}_2}_{\Tilde{\Gamma}_2^{DD}}(\lambda)}.
\end{equation}
Additionally, walking through the argument leading to equation \eqref{eq:79}, we find that 
\begin{equation}\label{eq:517}
\frac{u_1'(s_1,\lambda)}{u_1(s_1,\lambda)}=-\frac{E^{\Tilde{\Omega}_2}_{\Tilde{\Gamma}_2^{ND}}(\lambda)}{E^{\Tilde{\Omega}_2}_{\Tilde{\Gamma}_2^{DD}}(\lambda)}.
\end{equation}
We still need to classify the off diagonal entries of $\M_2$. By applying Theorem \ref{thm:general_ugamma} on $\Tilde{\Omega}_2$ with $i=1$, $j=2$, and $\Tilde{\Gamma}_2$ boundary conditions, we can find $u_2$ to be
\[
u_2(x_2,\lambda) = -\frac{z_{2,2}(x_2,\lambda)}{D_2(\lambda)}\frac{|\mathcal{C}_1(\lambda,\Vec{e}_2)|}{|C(\lambda)|}z_{1,1}'(s_1,\lambda),
\]
and the same theorem with $i=2$, $j=1$ tells us that $w_1$ is
\[
w_1(x_1,\lambda) = -\frac{z_{1,1}(x_1,\lambda)}{D_1(\lambda)}\frac{|\mathcal{C}_2(\lambda,\Vec{e}_1)|}{|C(\lambda)|}z_{2,2}'(s_2,\lambda).
\]
There is only one component of the trace vectors taken in Theorem \ref{thm:general_ugamma} due to similar reasoning as that in the start of the proof of Theorem \ref{thm:map_evans}. That is, because of the Dirichlet conditions of $\Tilde{\Gamma}_2$ at $s_1$ and $s_2$, we find that $\Vec{N}_1\Vec{e}_2 = -\Vec{e}_2$ and $\Vec{N}_2\Vec{e}_1 = -\Vec{e}_1$, and therefore $(\Vec{L}_1+\Vec{M}_1)\Vec{e}_2=(\Vec{L}_2+\Vec{M}_2)\Vec{e}_1=\Vec{0}$ (the reason why is nearly identical to that at the start of the proof of Theorem \ref{thm:map_evans}).

As in previous cases, we specify some well-behaved archetypal functions. Let $\phi_1,\theta_1$ on edge $\epsilon_1$ and $\phi_2,\theta_2$ on edge $\epsilon_2$ be components of solutions to $(H^\Omega-\lambda)\vu=\vec{0}$ which satisfy the following conditions:
\begin{align}\label{eq:180}
& \phi_1(s_1) = 0 & \theta_1(s_1) = 1 & &\phi_2(s_2) = 0 && \theta_2(s_2) = 1,\nonumber\\
& \phi_1'(s_1) = 1 & \theta_1'(s_1) = 0 && \phi_2'(s_2) = 1 && \theta_2'(s_2) = 0.
\end{align}
We observe that $z_{1,1}$ and $z_{2,2}$ equal $\phi_1$ and $\phi_2$, respectively. Thus, $D_1 = W(y_{\tau_1,1},\phi_1)$ and $D_2=W(y_{\tau_2,2},\phi_2)$. Recognizing that $u_1(s_1,\lambda)$ and $w_2(s_2,\lambda)$ are both specified to be $1$, the off diagonal terms in $\M_2(\lambda)$ are just $w_1'(s_1,\lambda)$ and $u_2'(s_2,\lambda)$. Taking the derivative of our current expressions, we obtain
\begin{equation}\label{eq:510}
u_2'(x_2,\lambda) = -\frac{\phi_2'(x_2,\lambda)}{D_2(\lambda)}\frac{|\mathcal{C}_1(\lambda,\Vec{e}_2)|}{|C(\lambda)|},
\end{equation}
and
\begin{equation}\label{eq:511}
w_1'(x_1,\lambda) = -\frac{\phi_1'(x_1,\lambda)}{D_1(\lambda)}\frac{|\mathcal{C}_2(\lambda,\Vec{e}_1)|}{|C(\lambda)|}.
\end{equation}
We know that by using Lemma \ref{thm:C_and_fundamental}, $|C(\lambda)|$ in both of these equations can be transformed into $(-1)^n|F(\vx,\lambda)|$ for a corresponding fundamental matrix $F$. By definition, $|F| = E^{\Tilde{\Omega}_2}_{\Tilde{\Gamma}_2^{DD}}$. One can use Lemma \ref{thm:C_and_fundamental} to see that $|\mathcal{C}_1(\lambda,\Vec{e}_2)|=(-1)^n|\mathcal{F}_1(\vx,\lambda,\Vec{e}_2)|$ where $\mathcal{F}_1(\vx,\lambda,\Vec{e}_2)$ is simply the fundamental matrix $F$ with its $\vz_1$ column replaced by a vector which is $y_{\tau_2,2}$ in the second wire row, $y_{\tau_2,2}'$ in the second wire derivative row, and zeros in all others. Similarly, $|\mathcal{C}_2(\lambda,\Vec{e}_1)|=(-1)^n|\mathcal{F}_2(\vx,\lambda,\Vec{e}_1)|$ where $\mathcal{F}_2(\vx,\lambda,\Vec{e}_1)$ is simply the fundamental matrix $F$ with its $\vz_2$ column replaced by a vector which is $y_{\tau_1,1}$ in the first wire row, $y_{\tau_1,1}'$ in the first wire derivative row, and zeros in all others.

Notice that by expanding down the $\vz_1$ and $\vz_2$ columns in $|\mathcal{F}_1(\vx,\lambda,\Vec{e}_2)|$ with $2\times 2$ blocks, we find that
\begin{equation}\label{eq:512}
|\mathcal{F}_1(\vx,\lambda,\Vec{e}_2)| = (-1)^q B_1 (\lambda)D_2(\lambda),
\end{equation}
where $B_1$ is the complementary minor and $(-1)^q$ is the sign of the defining permutation of this expansion term. Similarly, we can expand down the $\vz_1$ and $\vz_2$ columns in $|\mathcal{F}_2(\vx,\lambda,\Vec{e}_1)|$ with $2\times 2$ blocks to obtain
\begin{equation}\label{eq:513}
|\mathcal{F}_2(\vx,\lambda,\Vec{e}_1)| = (-1)^{q+2}B_2(\lambda)(-D_1(\lambda)),
\end{equation}
where $B_2$ is the complementary minor, and $(-1)^{q+2}$ again serves to denote the permutation sign. We can now rewrite our off-diagonal $u_2$ and $w_1$ entries as follows:
\begin{align}\label{eq:181}
u_2'(s_2,\lambda) &\myeq{\eqref{eq:510}} -\frac{\phi_2'(s_2,\lambda)}{D_2(\lambda)}\frac{|\mathcal{C}_1(\lambda,\Vec{e}_2)|}{|C(\lambda)|}\,\,\,\,\,\,\, \myeq{\text{Lemma }\ref{thm:C_and_fundamental}}\,\,\,\,\,\,\, -\frac{(-1)^n|\mathcal{F}_1(\vx,\lambda,\Vec{e}_2)|}{D_2(\lambda)(-1)^n|F(\vx,\lambda)|}\nonumber\\
&\myeq{\eqref{eq:512}} -\frac{(-1)^qB_1(\lambda)D_2(\lambda)}{D_2(\lambda)|F(\vx,\lambda)|} =\frac{(-1)^{q+1}B_1(\lambda)}{E^{\Tilde{\Omega}_2}_{\Tilde{\Gamma}_2^{DD}}(\lambda)}\nonumber\\
w_1'(s_1,\lambda) &\myeq{\eqref{eq:511}} -\frac{\phi_1'(s_1,\lambda)}{D_1(\lambda)}\frac{|\mathcal{C}_2(\lambda,\Vec{e}_1)|}{|C(\lambda)|} \,\,\,\,\,\,\, \myeq{\text{Lemma }\ref{thm:C_and_fundamental}}\,\,\,\,\,\,\, -\frac{(-1)^n|\mathcal{F}_2(\vx,\lambda,\Vec{e}_1)|}{D_1(\lambda)(-1)^n|F(\vx,\lambda)|} \nonumber\\
&\myeq{\eqref{eq:513}} -\frac{(-1)^{q+2}B_2(\lambda)(-D_1(\lambda))}{D_1(\lambda)|F(\vx,\lambda)|} =\frac{(-1)^{q+2}B_2(\lambda)}{E^{\Tilde{\Omega}_2}_{\Tilde{\Gamma}_2^{DD}}(\lambda)}.
\end{align}
Then the determinant of the two-sided map $\mathcal{M}_1+\mathcal{M}_2$ can be expanded out using our new expressions of the off-diagonal terms. Note that the objects we work with for the rest of this proof depend on $\lambda$ only, so we suppress this dependence to make formulas more readable. With this in mind, using Remark \ref{rem:twowireMs} and equations \eqref{eq:516}, \eqref{eq:517} and \eqref{eq:181}, the determinant expansion is
\begin{align}\label{eq:182}
&|\mathcal{M}_1+\mathcal{M}_2| =\begin{vmatrix}
M_1-\frac{E^{\Tilde{\Omega}_2}_{\Tilde{\Gamma}_2^{ND}}}{E^{\Tilde{\Omega}_2}_{\Tilde{\Gamma}_2^{DD}}} & \frac{(-1)^{q+2}B_2}{E^{\Tilde{\Omega}_2}_{\Tilde{\Gamma}_2^{DD}}}\\
\frac{(-1)^{q+1}B_1}{E^{\Tilde{\Omega}_2}_{\Tilde{\Gamma}_2^{DD}}} & \Tilde{M}_1 + \Tilde{M}_2 
\end{vmatrix}\nonumber\\
&= -\frac{E^{\Tilde{\Omega}_2}_{\Tilde{\Gamma}_2^{ND}}}{E^{\Tilde{\Omega}_2}_{\Tilde{\Gamma}_2^{DD}}}(\Tilde{M}_1 + \Tilde{M}_2)\nonumber+M_1(\Tilde{M}_1 + \Tilde{M}_2)+ \frac{B_1B_2}{(E^{\Tilde{\Omega}_2}_{\Tilde{\Gamma}_2^{DD}})^2}\nonumber\\
&\myeq{\eqref{eq:516}} -\frac{E^{\Tilde{\Omega}_2}_{\Tilde{\Gamma}_2^{ND}}}{E^{\Tilde{\Omega}_2}_{\Tilde{\Gamma}_2^{DD}}}\Tilde{M}_1 +\frac{E^{\Tilde{\Omega}_2}_{\Tilde{\Gamma}_2^{ND}}E^{\Tilde{\Omega}_2}_{\Tilde{\Gamma}_2^{DN}}}{(E^{\Tilde{\Omega}_2}_{\Tilde{\Gamma}_2^{DD}})^2} 
+M_1(\Tilde{M}_1 + \Tilde{M}_2)+ \frac{B_1B_2}{(E^{\Tilde{\Omega}_2}_{\Tilde{\Gamma}_2^{DD}})^2}.
\end{align}
Using Theorem \ref{thm:determinant_equivalence} adjusted to a subgraph $\Omega_2$, we can decompose $E^{\Omega_2}_{\Gamma_2'}$ as follows:
\begin{equation}\label{eq:515}
E^{\Omega_2}_{\Gamma_2'} = E^{\Tilde{\Omega}_1}_{\Tilde{\Gamma}_1} E^{\Tilde{\Omega}_2}_{\Tilde{\Gamma}_2^{ND}}(\hat{M}_1+\hat{M}_2),
\end{equation}
where $\hat{M}_1+\hat{M}_2$ is the relevant two-sided map associated with $s_2$. Note that $\Hat{M}_1=\Tilde{M}_1$. Additionally, using Theorem \ref{thm:map_evans} adjusted to a subgraph $\tilde\Omega_2$, we can rewrite $\Hat{M}_2$ as

\begin{equation}\label{eq:hM}
\Hat{M}_2 = -\frac{E^{\Tilde{\Omega}_2}_{\Tilde{\Gamma}_2^{NN}}}{E^{\Tilde{\Omega}_2}_{\Tilde{\Gamma}_2^{ND}}}.
\end{equation}

Therefore, we can expand out $(M_1+M_2)(\Tilde{M}_1 + \Tilde{M}_2)$ as
\begin{align}\label{eq:183}
&(M_1+M_2)(\Tilde{M}_1 + \Tilde{M}_2) = M_1(\Tilde{M}_1 + \Tilde{M}_2) - \frac{E^{\Omega_2}_{\Gamma_2'}}{E^{\Omega_2}_{\Gamma_2}}(\Tilde{M}_1 + \Tilde{M}_2) \nonumber\\
& \myeq{\eqref{eq:501},\eqref{eq:515}}\,\,\,\, M_1(\Tilde{M}_1 + \Tilde{M}_2) - \frac{E^{\Tilde{\Omega}_1}_{\Tilde{\Gamma}_1}E^{\Tilde{\Omega}_2}_{\Tilde{\Gamma}_2^{ND}}(\Hat{M}_1+\Hat{M}_2)}{E^{\Tilde{\Omega}_1}_{\Tilde{\Gamma}_1}E^{\Tilde{\Omega}_2}_{\Tilde{\Gamma}_2^{DD}}(\Tilde{M}_1+\Tilde{M}_2)}(\Tilde{M}_1 + \Tilde{M}_2)\nonumber\\
& \myeq{\eqref{eq:hM}}\,\,\,\, M_1(\Tilde{M}_1 + \Tilde{M}_2) - \frac{E^{\Tilde{\Omega}_2}_{\Tilde{\Gamma}_2^{ND}}}{E^{\Tilde{\Omega}_2}_{\Tilde{\Gamma}_2^{DD}}}\Tilde{M}_1 + \frac{E^{\Tilde{\Omega}_2}_{\Tilde{\Gamma}_2^{ND}}}{E^{\Tilde{\Omega}_2}_{\Tilde{\Gamma}_2^{DD}}}\frac{E^{\Tilde{\Omega}_2}_{\Tilde{\Gamma}_2^{NN}}}{E^{\Tilde{\Omega}_2}_{\Tilde{\Gamma}_2^{ND}}}\nonumber\\
&=M_1(\Tilde{M}_1 + \Tilde{M}_2) - \frac{E^{\Tilde{\Omega}_2}_{\Tilde{\Gamma}_2^{ND}}}{E^{\Tilde{\Omega}_2}_{\Tilde{\Gamma}_2^{DD}}}\Tilde{M}_1 + \frac{E^{\Tilde{\Omega}_2}_{\Tilde{\Gamma}_2^{NN}}}{E^{\Tilde{\Omega}_2}_{\Tilde{\Gamma}_2^{DD}}}.
\end{align}
Comparing the terms in equations \eqref{eq:182} and \eqref{eq:183}, we see that we have the desired equality if
\begin{equation}\label{eq:184}
\frac{E^{\Tilde{\Omega}_2}_{\Tilde{\Gamma}_2^{NN}}}{E^{\Tilde{\Omega}_2}_{\Tilde{\Gamma}_2^{DD}}} = \frac{B_1B_2}{(E^{\Tilde{\Omega}_2}_{\Tilde{\Gamma}_2^{DD}})^2}+\frac{E^{\Tilde{\Omega}_2}_{\Tilde{\Gamma}_2^{ND}}E^{\Tilde{\Omega}_2}_{\Tilde{\Gamma}_2^{DN}}}{(E^{\Tilde{\Omega}_2}_{\Tilde{\Gamma}_2^{DD}})^2}.
\end{equation}
An equivalent condition can be obtained by rewriting the above equation as
\begin{equation}\label{eq:185}
E^{\Tilde{\Omega}_2}_{\Tilde{\Gamma}_2^{DD}}E^{\Tilde{\Omega}_2}_{\Tilde{\Gamma}_2^{NN}}-E^{\Tilde{\Omega}_2}_{\Tilde{\Gamma}_2^{ND}} E^{\Tilde{\Omega}_2}_{\Tilde{\Gamma}_2^{DN}} = B_1B_2.
\end{equation}
Note that, since every determinant in this proposed equality is being multiplied by another determinant with an underlying matrix of the same shape, we can prove this statement for slightly modified matrices. That is, we can rearrange rows and columns as long as the same swaps are performed for every determinant involved. For the remainder of this proof, we let all matrices involved have row ordering of the form $y_1,y_1',y_2,y_2',...,y_n,y_n'$ instead of the standard $y_1,y_2,...,y_n,y_1',y_2',...,y_n'$. These row swapped objects will be denoted by calligraphic letters. For example, $E^{\Tilde{\Omega}_2}_{\Tilde{\Gamma}_2^{DD}}$ and $B_1$ would be rewritten as $\E^{\Tilde{\Omega}_2}_{\Tilde{\Gamma}_2^{DD}}$ and $\B_1$, respectively. Therefore, equation \eqref{eq:185} is equivalent to
\begin{equation}\label{eq:cBB}
\E^{\Tilde{\Omega}_2}_{\Tilde{\Gamma}_2^{DD}}\E^{\Tilde{\Omega}_2}_{\Tilde{\Gamma}_2^{NN}}-\E^{\Tilde{\Omega}_2}_{\Tilde{\Gamma}_2^{ND}} \E^{\Tilde{\Omega}_2}_{\Tilde{\Gamma}_2^{DN}}=\B_1\B_2.
\end{equation}

We now begin a brute expansion of the various Evans functions on the left hand side of equation \eqref{eq:cBB}. Note that they are entirely identical except in the $\vz_1$ and $\vz_2$ columns. In these columns, playing the role of $z_1$ and $z_2$ (and thus their derivatives), are: $\phi_1$ and $\phi_2$ for $\E^{\Tilde{\Omega}_2}_{\Tilde{\Gamma}_2^{DD}}$, $\phi_1$ and $\theta_2$ for $\E^{\Tilde{\Omega}_2}_{\Tilde{\Gamma}_2^{DN}}$, $\theta_1$ and $\phi_2$ for $\E^{\Tilde{\Omega}_2}_{\Tilde{\Gamma}_2^{ND}}$, and $\theta_1$ and $\theta_2$ for $\E^{\Tilde{\Omega}_2}_{\Tilde{\Gamma}_2^{NN}}$. 

We choose to expand each Evans function along these two columns in $2\times 2$ blocks. This means that, letting $f_i$ stand in for the $\phi_i$ or $\theta_i$ in the $\vz_i$ column, these determinants expand in the form
\begin{align}\label{eq:195}
&-\begin{vmatrix}
f_1 & 0\\
0 &f_2
\end{vmatrix}M_{2,4} + \begin{vmatrix}
f_1 & 0\\
0 &f_2'
\end{vmatrix}M_{2,3} + \begin{vmatrix}
f_1' & 0\\
0 & f_2
\end{vmatrix}M_{1,4} - \begin{vmatrix}
f_1' & 0\\
0 & f_2'
\end{vmatrix}M_{1,3} \nonumber\\
&=-f_1f_2M_{2,4} + f_1f_2'M_{2,3} + f_1'f_2M_{1,4} - f_1'f_2'M_{1,3},
\end{align}
where each $M_{j,k}$ is some complementary minor. The $j$ and $k$ indicate which two of the first four rows are inherited as the first two rows of the undelying matrix of minor. We note that, using such notation, $\B_1\B_2=M_{1,2}M_{3,4}$ (defined the same way as the other $M_{j,k}$). 

Now, using the expansion from equation \eqref{eq:195}, we can rewrite equation \eqref{eq:cBB} as
\begin{align}\label{eq:200}
&\E^{\Tilde{\Omega}_2}_{\Tilde{\Gamma}_2^{DD}}\E^{\Tilde{\Omega}_2}_{\Tilde{\Gamma}_2^{NN}}-\E^{\Tilde{\Omega}_2}_{\Tilde{\Gamma}_2^{ND}} \E^{\Tilde{\Omega}_2}_{\Tilde{\Gamma}_2^{DN}}\nonumber
\\&= (-\phi_1\phi_2M_{2,4} + \phi_1\phi_2'M_{2,3} + \phi_1'\phi_2M_{1,4} - \phi_1'\phi_2'M_{1,3})\nonumber
\\&\times (-\theta_1\theta_2M_{2,4} + \theta_1\theta_2'M_{2,3} + \theta_1'\theta_2M_{1,4} - \theta_1'\theta_2'M_{1,3})\nonumber
\\&- (-\theta_1\phi_2M_{2,4} + \theta_1\phi_2'M_{2,3} + \theta_1'\phi_2M_{1,4} - \theta_1'\phi_2'M_{1,3})\nonumber
\\&\times 
(-\phi_1\theta_2M_{2,4} + \phi_1\theta_2'M_{2,3} + \phi_1'\theta_2M_{1,4} - \phi_1'\theta_2'M_{1,3})\nonumber
\\&=(-\theta_1\phi_2\phi_1'\theta_2' + \phi_1\phi_2\theta_1'\theta_2' - \theta_1'\phi_2'\phi_1\theta_2 + \phi_1'\phi_2'\theta_1\theta_2)M_{1,3}M_{2,4}\nonumber
\\&+(-\theta_1\phi_2'\phi_1'\theta_2 + \phi_1\phi_2'\theta_1'\theta_2 - \theta_1'\phi_2\phi_1\theta_2' + \phi_1'\phi_2\theta_1\theta_2')M_{1,4}M_{2,3} \nonumber
\\&= \begin{vmatrix}
\theta_1 & \phi_1\\
\theta_1' & \phi_1'
\end{vmatrix}\begin{vmatrix}
\theta_2 & \phi_2\\
\theta_2' & \phi_2'
\end{vmatrix}(M_{1,3}M_{2,4}-M_{1,4}M_{2,3})\nonumber
\\&= M_{1,3}M_{2,4}-M_{1,4}M_{2,3}.
\end{align}
We can now prove equation \eqref{eq:cBB} by showing that $M_{1,3}M_{2,4}-M_{1,4}M_{2,3}= M_{1,2}M_{3,4}$. This can be done via the construction of a special matrix $A$. $A$ is a $(4n-4)\times (4n-4)$ matrix, in which each row is split in two halves, each half being a $2n-2$ zero row or a $2n-2$ length row from fundamental matrix $F$ with its $\vz_1$ and $\vz_2$ columns removed. Let $R_1$ represent the first row of this modified $F$, $R_2$ the second, and so on down to $R_{2n}$. Then, listing zero rows as $0$, we construct $A$ like so:
\begin{equation}\label{eq:196}
A = \begin{bmatrix}
R_1 & 0\\
R_2 & R_2\\
R_3 & R_3\\
R_4 & R_4\\
R_5 & 0\\
\vdots & \vdots\\
R_{2n} & 0\\
0 & R_5\\
\vdots & \vdots\\
0 & R_{2n}
\end{bmatrix}.
\end{equation}
The rows from $R_5$ to $R_{2n}$ are those which are common in all the various $M_{i,j}$. Thus, the relevant $M_{i,j}$ minors can be written as
\begin{equation}\label{eq:205}
M_{i,j} = \begin{vmatrix}
R_i\\R_j\\R_5\\\vdots\\R_{2n}
\end{vmatrix}.
\end{equation}
By expanding in $(2n-2)\times (2n-2)$ blocks down the first $2n-2$ columns, we arrive at the following expression:
\begin{equation}\label{eq:198}
|A| = M_{1,2}M_{3,4}-M_{1,3}M_{2,4}+M_{1,4}M_{2,3}.
\end{equation}
Alternatively, we can do some row and column operations before performing this expansion to alter $A$ without changing its determinant. First, for the $j$th column of $A$ (where $1\leq j\leq 2n-2$), subtract the $(2n-2+j)$th column. Next, for each $k$ such that $2n-4\leq k\leq 4n-4$, add the $(k-2n+4)$th row to the $k$th row. So, the determinant $|A|$ can now be written as
\begin{equation}\label{eq:197}
|A| = \begin{vmatrix}
R_1 & 0\\
0 & R_2\\
0 & R_3\\
0 & R_4\\
R_5 & 0\\
\vdots & \vdots\\
R_{2n} & 0\\
0 & R_5\\
\vdots & \vdots\\
0 & R_{2n}
\end{vmatrix}.
\end{equation}
This determinant is zero, since we choose to expand in $(2n-2)\times (2n-2)$ blocks, but there are only $2n-3$ nonzero rows to be selected for minors in the left column collection: $R_1,R_5,...,R_{2n}$. Thus, every term in the expansion will have a zero row, and will be $0$. So, by our two calculations of $|A|$, we have
\[
M_{1,2}M_{3,4}-M_{1,3}M_{2,4}+M_{1,4}M_{2,3} =0,
\]
which implies
\[
M_{1,2}M_{3,4}=M_{1,3}M_{2,4}-M_{1,4}M_{2,3} \myeq{\eqref{eq:200}} \E^{\Tilde{\Omega}_2}_{\Tilde{\Gamma}_2^{DD}}\E^{\Tilde{\Omega}_2}_{\Tilde{\Gamma}_2^{NN}}-\E^{\Tilde{\Omega}_2}_{\Tilde{\Gamma}_2^{ND}} \E^{\Tilde{\Omega}_2}_{\Tilde{\Gamma}_2^{DN}}.
\]
Since $M_{1,2}M_{2,4} = \B_1\B_2$, the above equation gives us the following
\[
\B_1\B_2 = \E^{\Tilde{\Omega}_2}_{\Tilde{\Gamma}_2^{DD}}\E^{\Tilde{\Omega}_2}_{\Tilde{\Gamma}_2^{NN}}-\E^{\Tilde{\Omega}_2}_{\Tilde{\Gamma}_2^{ND}} \E^{\Tilde{\Omega}_2}_{\Tilde{\Gamma}_2^{DN}},
\]
which proves equation \eqref{eq:185}, thus completing the proof.
\end{proof}

Theorems \ref{thm:twosplitonewire} and \ref{thm:twosplittwowire} are the two cases making up the proof of Theorem \ref{thm:twoCutEvans}. The straightforward relationship between the counting functions in Theorem \ref{thm:counting2d} and the Evans functions and maps in Theorem \ref{thm:twoCutEvans} means that a proof for the counting relationship immediately follows.

\section{Applications}
Here we will consider a quantum star graph with two edges of length $1$. We will impose the Kirchhoff conditions $u_1(0)=u_2(0)$ and $u_1'(0)+u_2'(0)=0$ at the origin. We wish to see how the eigenvalues change as various potentials and endpoint conditions are introduced. In particular, we will graph Evans functions and map curves as $\lambda$ varies over a certain interval. Since our eigenvalue counting formulas follow directly from our Evans function equalities, they are true over any interval $[\lambda_1,\lambda_2]$. Thus, we represent the count of eigenvalues of an operator or zeros minus poles of a two sided map for $\lambda\in [\lambda_1,\lambda_2]$ by passing the interval as an argument in the relevant counting function from Theorem \ref{thm:counting} or \ref{thm:counting2d}. For example, $\mathcal{N}_X([\lambda_1,\lambda_2])$ equals the number of eigenvalues of $X$ for $\lambda\in[\lambda_1,\lambda_2]$.

\subsection{Potential Barrier/Well at Wire's End}
A potential barrier or well is a localized constant potential at the end of one wire. We can model these by setting $V_i(x_i)=0$ for all potentials in $H^\Omega$ except when $i=1$, and letting $V_1(x_1)$ be defined piecewise as:
\begin{equation}\label{eq:100}
V_1(x_1) = \begin{cases}
\nu, & s_1\leq x_1\leq 1\\
0, & 0\leq x_1 < s_1
\end{cases},
\end{equation}
where $\nu\in \RR$. In this example, the boundary conditions $\Gamma$ will be represented by the following $\alpha$ and $\beta$ matrices:
\begin{align}\label{eq:101}
\beta_1 = I_2,\,\, \beta_2 = 0_2,\,\,
\alpha_1 = \begin{bmatrix}
1 & -1\\
0 & 0
\end{bmatrix},\,\, \alpha_2 =\begin{bmatrix}
0 & 0 \\
1 & 1
\end{bmatrix}.
\end{align}
These are Kirchhoff at the origin, and Dirichlet at the outer vertices. Using Theorem \ref{thm:determinant_equivalence}, we can count the eigenvalues of $H^{\Omega}_{\Gamma}$ via Evans functions of its subproblems.

First, we split the problem in two by placing a new vertex at $x_1=s_1$ and imposing a Dirichlet condition there. The operator $H^{\Omega_1}_{\Gamma_1}$ defined over $\Omega_1$ is a Dirichlet Laplacian (with nonzero potential $\nu$), while the operator $H^{\Omega_2}_{\Gamma_2}$ over $\Omega_2$ is still a Kirchhoff Laplacian, but with a smaller edge $\epsilon_1$.

To find the eigenvalues of $H^{\Omega_1}_{\Gamma_1}$, we construct the corresponding Evans function. For $\lambda-\nu>0$, it is equal to
\begin{equation}\label{eq:103}
E^{\Omega_1}_{\Gamma_1}(\lambda) = \frac{\sin((1-s_1)\sqrten)}{\sqrten}.
\end{equation}

To find the eigenvalues of $H^{\Omega_2}_{\Gamma_2}$, we will construct the Evans function $E^{\Omega_2}_{\Gamma_2}$. For $\lambda>0$, we have
\begin{align}\label{eq:104}
E^{\Omega_2}_{\Gamma_2}(\lambda) =-\frac{\sin(\rlam(1+s_1))}{\rlam}\,.
\end{align}

Next we solve for the one-sided maps $M_1$ and $M_2$. For $\lambda>0$, we have
\begin{equation}\label{eq:106}
M_1(\lambda) = \frac{E^{\Omega_1}_{\Gamma_1'}(\lambda)}{E^{\Omega_1}_{\Gamma_1}(\lambda)}= \frac{\sqrt{\lambda-\nu}\cos(\sqrt{\lambda-\nu}(s_1-1))}{\sin(\sqrt{\lambda-\nu}(s_1-1))}\,,
\end{equation}

Finally, for $\lambda>0$ we have
\begin{equation}\label{eq:108}
M_2(\lambda) = -\frac{E^{\Omega_2}_{\Gamma_2'}(\lambda)}{E^{\Omega_2}_{\Gamma_2}(\lambda)}= -\frac{\rlam\cos(\rlam(1+s_1))}{\sin(\rlam (1+s_1))}.
\end{equation}

Our method so far has been quite simple, consisting of a few determinants and some easy initial value problems defined by the various boundary condition matrices. We would like to check our results against the actual eigenvalues of $H^\Omega_\Gamma$. It can be shown that the positive eigenvalues of $H^\Omega_\Gamma$ are precisely the zeros of the function
\begin{align}\label{eq:110}
\begin{split}
E(\lambda)=&\rlam \cos(\rlam (1+s_1))\sin(\sqrten(s_1-1))
\\&-\sqrten\cos(\sqrten(s_1-1))\sin(\rlam(1+s_1)),
\end{split}
\end{align}
Note that $E(\lambda)$ is an Evans function for $H^{\Omega}_{\Gamma}$, although not necessarily the one specified by our usual construction.

One can graphically see (for example in Figure \ref{fig:exampleSingle}) that the eigenvalues of $H^{\Omega}_\Gamma$ are counted accurately by plotting $E$, $M_1+M_2$, $E^{\Omega_1}_{\Gamma_1}$, and $E^{\Omega_2}_{\Gamma_2}$.

\begin{figure}[H]
    \centering
    \includegraphics[width=4in]{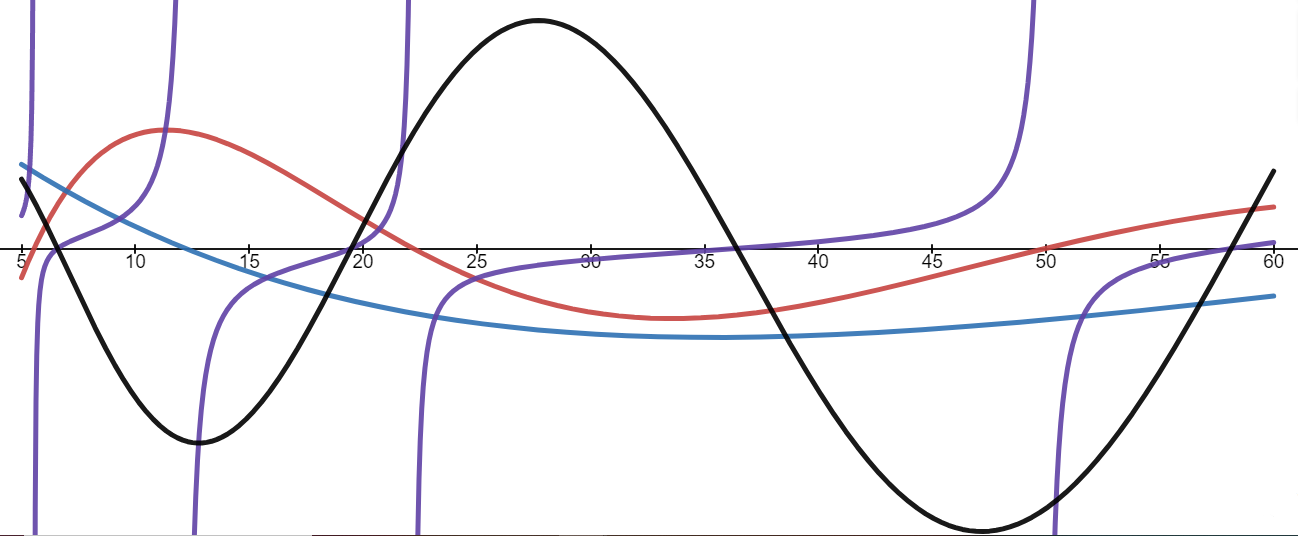}
    \caption{A single split with $s_1=\frac{1}{3}$ and $\nu = -10$ for $\lambda\in [5,60]$. The red curve is $10E^{\Omega_2}_{\Gamma_2}$, the blue is $15E^{\Omega_1}_{\Gamma_1}$, the purple curve is $\frac{1}{25}(M_1 +M_2)$, and the black curve is $E$. These objects are vertically rescaled to accentuate the zeros and poles.}
    \label{fig:exampleSingle}
\end{figure}

In particular, the zeros of $E^{\Omega_1}_{\Gamma_1}$ and $E^{\Omega_2}_{\Gamma_2}$ coincide with poles of $M_1+M_2$, representing cancellation. This is expected, since $E^{\Omega}_{\Gamma}$ has no zeros at these points. Also as expected, the zeros of $M_1+M_2$ coincide with those of $E^{\Omega}_{\Gamma}$. In terms of counting functions, we find that
\begin{align*}
\mathcal{N}_{H^\Omega_\Gamma}[5,60]  = 4,\hspace{.2in}
\mathcal{N}_{H^{\Omega_1}_{\Gamma_1}}[5,60] = 1,\hspace{.2in}
\mathcal{N}_{H^{\Omega_2}_{\Gamma_2}}[5,60] = 3,\\
N[5,60] = 4-4 = 0\hspace{1in}
\end{align*}
So the counting functions work in this case as predicted.

\subsection{Potential Barrier/Well on Wire's Interior}
In this example, we suppose that the boundary conditions at $x_1=1$ and $x_2=1$ are both Neumann ($u_i'(1) = 0$) for $i=1,2$. We additionally suppose we have a potential localized on the interior of wire one. Define the nonzero potential $V_1$ in $H^{\Omega}_\Gamma$ as:
\[
V_1(x_1) = \begin{cases}
\nu, & x_1\in (s_2,s_1)\\
0, & x_1\not\in (s_2,s_1)
\end{cases}
\]
where $0< s_2<s_1<1$, and $V_2(x_2)=0$. We wish to apply Theorem \ref{thm:counting2d} to count the eigenvalues of $H^{\Omega}_\Gamma$. For $\lambda>0$, we find that the relevant Evans functions are:
\[
E^{\Omega_1}_{\Gamma_1}(\lambda) = -\cos(\rlam(1-s_1)),
\]
\[
E^{\Tilde{\Omega}_1}_{\Tilde{\Gamma}_1}(\lambda) = \frac{\sin(\sqrten(s_1-s_2))}{\sqrten},
\]
\[
E^{\Tilde{\Omega}_2}_{\Tilde{\Gamma}_2}(\lambda) = -\cos(\rlam(1+s_2)),
\]
and the relevant one-sided maps are:
\[
\mathcal{M}_1(\lambda) = \begin{bmatrix}
\frac{\rlam\sin(\rlam (s_1-1))}{\cos(\rlam (s_1-1))} & 0\\ 0 & -\frac{\rlam \sin(\rlam (s_2+1))}{\cos(\rlam (s_2+1))}
\end{bmatrix}
\]
and
\[
\mathcal{M}_2(\lambda) = \begin{bmatrix}
\frac{\sqrten\cos(\sqrten(s_1-s_2))}{\sin(\sqrten (s_1-s_2))} & \frac{\sqrten}{\sin(\sqrten (s_2-s_1))}\\ -\frac{\sqrten}{\sin(\sqrten (s_1-s_2))} & -\frac{\sqrten\cos(\sqrten(s_2-s_1))}{\sin(\sqrten (s_2-s_1))}
\end{bmatrix}.
\]
Therefore, the map determinant is equal to
\begin{align*}
|(\M_1+\M_2)(\lambda)| =& \left( \rlam \tan(\rlam(s_1-1))+\sqrten\cot(\sqrten(s_1-s_2)) \right)
\\ &\times\left( -\rlam \tan(\rlam (s_2+1)) - \sqrten \cot(\sqrten (s_2-s_1)) \right)
\\& -\left (\frac{\sqrten}{\sin(\sqrten (s_1-s_2))} \right )^2
\end{align*}
Finally, one can show that the eigenvalues of $H^{\Omega}_\Gamma$ are given by the zeros of the following curve:
\begin{align}\label{eq:400}
E(\lambda)=& (2\lambda-\nu)\cos(\rlam(2+s_1-s_2))\sin(\sqrten(s_1-s_2))\nonumber \\&+ \nu\cos(\rlam(s_1+s_2))\sin(\sqrten(s_1-s_2))\nonumber\\&-2\rlam\sqrten\cos(\sqrten(s_1-s_2))\sin(\rlam(2+s_1-s_2)),
\end{align}
which, as in the previous example, is an Evans function for $H^{\Omega}_{\Gamma}$.

Again, we can check our result visually by plotting these objects against one another and making sure the poles and zeros line up in the right places. An example is shown in Figure \ref{fig:oneTwoExample}.
\begin{figure}[H]
    \centering
    \includegraphics[width=4in]{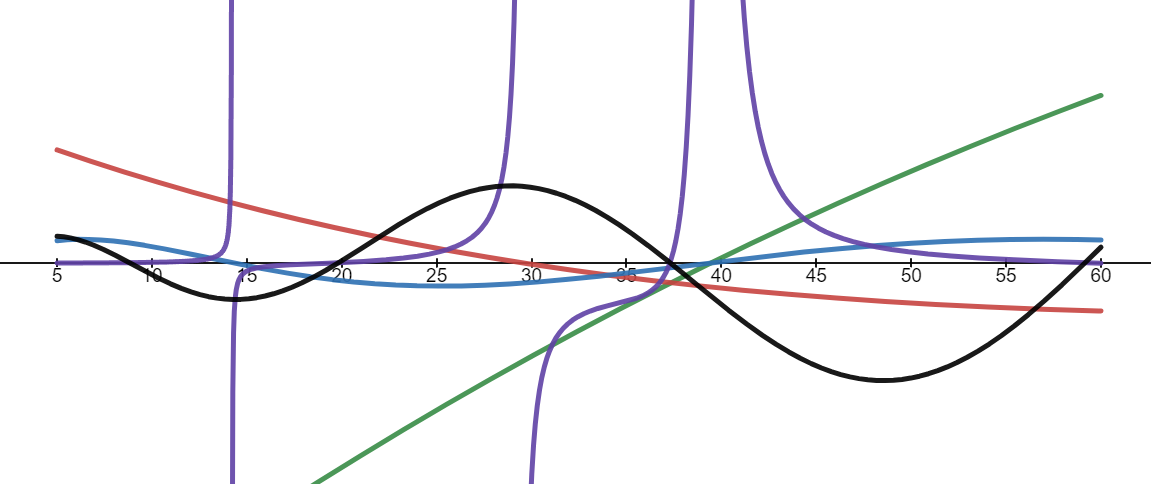}
    \caption{Two splits on one wire with $s_1=\frac{3}{4}$, $s_2=\frac{1}{4}$, and $\nu = -10$. The green curve is $20E^{\Omega_1}_{\Gamma_1}$, the red is $20E^{\Tilde{\Omega}_1}_{\Tilde{\Gamma}_1}$, the blue is $E^{\Tilde{\Omega}_2}_{\Tilde{\Gamma}_2}$, the purple curve is $\frac{1}{1000}|\M_1 +\M_2|$, and the black curve is $\frac{1}{10}E$. These objects are vertically rescaled to accentuate the zeros and poles.}
    \label{fig:oneTwoExample}
\end{figure}
This figure not only has poles and zeros lining up in the right places, it also has proper multiplicity: order one poles when only one Evans function has a zero of order one, and order two poles when two Evans functions have a zero of order one at the same $\lambda$. Indeed, we find that
\begin{align*}
\mathcal{N}_{H^\Omega_\Gamma}[5,60]  = 4,\hspace{.2in}
\mathcal{N}_{H^{\Omega_1}_{\Gamma_1}}[5,60] = 1,\hspace{.2in}
\mathcal{N}_{H^{\Tilde{\Omega}_1}_{\Tilde{\Gamma}_1}}[5,60]  = 1,\hspace{.2in}
\mathcal{N}_{H^{\Tilde{\Omega}_2}_{\Tilde{\Gamma}_2}}[5,60]  = 2,\\
N[5,60] = 4-4 = 0.\hspace{1.4in}
\end{align*}

\subsection{Potential Over Multiple Wires}
To demonstrate our last result, we consider a potential extending over both wires, and suppose that the boundary conditions at $x_1=1$ and $x_2=1$ are Dirichlet. Let
\[
V_i(x_i) = \begin{cases}
\nu, & x_i\in (0,s_i)\\
0, & x_i\not\in (0,s_i)
\end{cases}
\]
for both $i=1$ and $i=2$. To apply our theorem, we must construct the corresponding Evans functions.
\[
E^{\Omega_1}_{\Gamma_1}(\lambda) = -\frac{\sin(\rlam(1-s_1))}{\rlam},
\]
\[
E^{\Tilde{\Omega}_1}_{\Tilde{\Gamma}_1}(\lambda) = -\frac{\sin(\rlam(1-s_2))}{\rlam},
\]
and
\[
E^{\Tilde{\Omega}_2}_{\Tilde{\Gamma}_2}(\lambda) = -\frac{\sin(\sqrten(s_1+s_2))}{\sqrten}.
\]
It can be shown that the determinant of $\M_1+\M_2$ is given by
\begin{align*}
|(\M_1+\M_2)(\lambda)| =& \left(\sqrt{x+10}\cot\left(\sqrt{x+10}\right)+\sqrt{x}\cot\left(\frac{\sqrt{x}}{2}\right)+\frac{\sqrt{x+10}}{\sin\left(\sqrt{x+10}\right)}\right)
\\\times&\left(\sqrt{x+10}\cot\left(\sqrt{x+10}\right)+\sqrt{x}\cot\left(\frac{\sqrt{x}}{2}\right)-\frac{\sqrt{x+10}}{\sin\left(\sqrt{x+10}\right)}\right)
\end{align*}

It can also be shown that the function
\begin{align*}
E(\lambda) =& \sqrt{x}\sqrt{10+x}\cos\left(\sqrt{10+x}\right)\sin\left(\sqrt{x}\right)
\\+&\left(-5+\left(5+x\right)\cos\left(\sqrt{x}\right)\right)\sin\left(\sqrt{10+x}\right)
\end{align*}
indicates the location of the positive eigenvalues of $H^{\Omega}_\Gamma$. Once again, one can use a plot to confirm our formula accurately predicts the location of eigenvalues, as demonstrated in Figure \ref{fig:exampleDouble}.
\begin{figure}[H]
    \centering
    \includegraphics[width=4in]{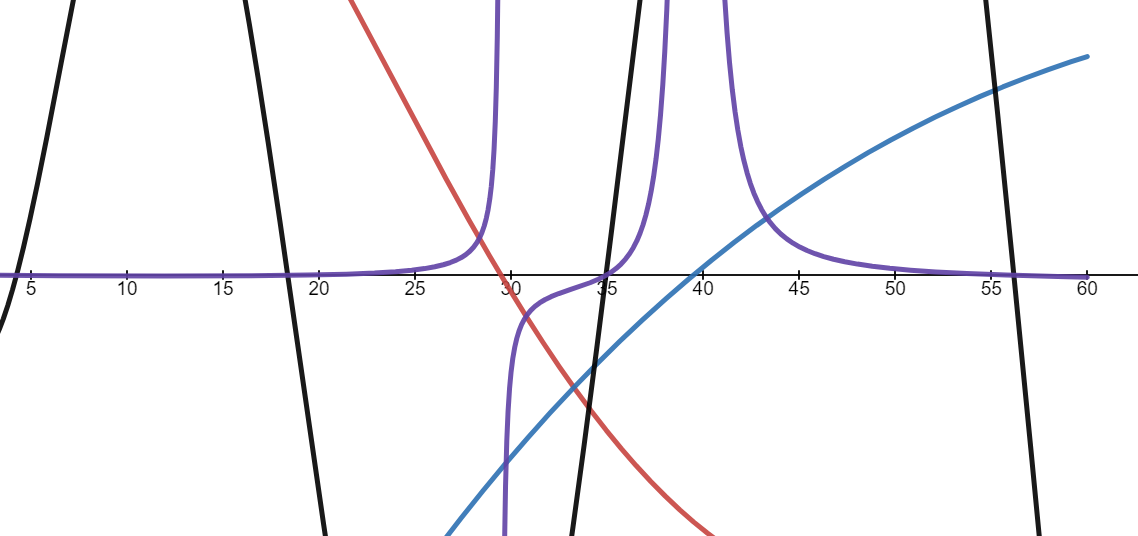}
    \caption{Two splits on two distinct wires with $s_1=s_2=\frac{1}{2}$ and $\nu = -10$. The blue curve is $100E^{\Omega_1}_{\Gamma_1}$ as well as $100E^{\Tilde{\Omega}_1}_{\Tilde{\Gamma}_1}$ (they are identical in this case), the red is $100E^{\Tilde{\Omega}_2}_{\Tilde{\Gamma}_2}$, the purple curve is $10|\M_1 +\M_2|$, and the black curve is $E$. These objects are vertically rescaled to accentuate the zeros and poles.}
    \label{fig:exampleDouble}
\end{figure}
Again, this figure demonstrates proper cancellation between the subproblem eigenvalues and the poles of the map determinant, and the proper correspondence between eigenvalues of $H^{\Omega}_{\Gamma}$ and the zeros of the map determinant. Examining the counting functions, we find that
\begin{align*}
\mathcal{N}_{H^\Omega_\Gamma}[3,60]  = 4,\hspace{.2in}
\mathcal{N}_{H^{\Omega_1}_{\Gamma_1}}[3,60]  = 1,\hspace{.2in}
\mathcal{N}_{H^{\Tilde{\Omega}_1}_{\Tilde{\Gamma}_1}}[3,60]  = 1,\hspace{.2in}
\mathcal{N}_{H^{\Tilde{\Omega}_2}_{\Tilde{\Gamma}_2}}[3,60]  = 1,\\
N[5,60] = 4-3 = 1.\hspace{1.4in}
\end{align*}

\appendix
\section{Note on Linear Algebraic Concepts}
Despite introducing separate spatial variables on each wire, we still wish to often refer to notions of linear dependence and independence of solution sets. 
\begin{theorem}\label{thm:Abel}
Let $\{\vu_1,...,\vu_{2n}\}$ be a set of solutions to the equation $(H^\Omega-\lambda)\vu = \Vec{0}$, where $H^\Omega$ is the differential expression from \eqref{eq:2}. Let $U$ be a $2n\times 2n$ matrix whose $i$th column is $\vu_i$ followed by $\vu_i'$. Then $|U|$ is a function of $\lambda$ only.
\end{theorem}
\begin{proof}
We can calculate $|U|$ using Laplace expansion by complementary minors. In particular, we can choose to expand along the pair of rows $1$ and $n+1$, which contain information about $(\vu_i)_1$ and $(\vu_i)_1'$, respectively. Let $H$ be the set of two-element combinations $\{s,t\}$ from the set $\{1,...,2n\}$. Then the determinant of $U$ is
\begin{equation}\label{eq:det_expansion}
|U(\lambda)| = (-1)^{n-1}\sum_{\{s,t\}\in H} \pi(s,t)\begin{vmatrix}
(\vu_s)_1(x_1,\lambda) & (\vu_t)_1(x_1,\lambda)\\
(\vu_s)_1'(x_1,\lambda) & (\vu_t)_1'(x_1,\lambda)
\end{vmatrix}c_{s,t} \, ,
\end{equation}
where $\pi(s,t)$ is the sign of a permutation sending $1$ to $s$, $2$ to $t$, and all other numbers in $\{1,...,2n\}\setminus \{1,2\}$ are sent to the elements of $\{1,...,2n\} \setminus\{s,t\}$ in increasing order, and $c_{s,t}$ is the $(2n-2)\times (2n-2)$ complementary minor obtained from $U$ by deleting the $1$st and $(n+1)$th rows, and the $s$th and $t$th columns. The $(-1)^{n-1}$ term represents the row permutation  sending $1$ to $1$, $n+1$ to $2$, and all else paired in increasing order. Since the $2\times 2$ determinants involved in this expansion are solutions to a second order ordinary differential equation with no first order term, by Abel's Identity they depend only on $\lambda$, and not on the choice of $x_1$. We can clearly use a similar method to expand each $|c_{s,t}|$ by $2\times 2$ blocks of functions sharing one spatial variable, so none of these terms depend on $\vx$ either. Thus $|U|$ is $\vx$ independent, just as in the single variable case.
\end{proof}

As with a set of single variable vector functions, we say that $\{\vu_1,...,\vu_{2n}\}$ forms a linearly independent set if the only solution to
\begin{equation}\label{eq:independence}
\sum_{i=1}^{2n} c_i\vu_i(\vx,\lambda) = \Vec{0}\, ,
\end{equation}
is the one with $c_i=0$ for all $i$. This definition needs no adjustment from the single-variable case since vector addition acts componentwise, and thus the separate spatial variables do not interact at all in this sum. Indeed, we are free to vary each $x_i$ individually without altering the result of the sum. Since the typical definition of linear dependence applies for functions on a quantum graph, and since Abel's Identity still holds, one can use the standard proof to show that $|U|=0$ if and only if there is a linearly dependent set of columns in $U$.

\subsection*{Acknowledgements}
A.S. acknowledges the support of an AIM workshop on \emph{Computer assisted proofs for stability analysis of nonlinear waves}, where some of the key ideas were discussed.
A.S. acknowledges support from the National Science Foundation
under grant DMS-1910820.
\printbibliography
\end{document}